\documentclass[12pt,reqno]{amsart}
\usepackage{amsmath,amsfonts,amsthm,amssymb,color}
\usepackage[usenames,dvipsnames]{xcolor}
\usepackage[T1]{fontenc}
\usepackage{enumerate}
\usepackage{mathrsfs}
\usepackage{hyperref}
\hypersetup{
     colorlinks   = true,
     citecolor    = blue,
     linkcolor    = blue
}

\usepackage{pdfsync}{\tiny }

\usepackage[left=1in, right=1in, top=1.1in,bottom=1.1in]{geometry}
\numberwithin{equation}{section}

\def\R{\mathbb{R}}

\def\wt{\widetilde}
\def\wh{\widehat}

\def\Z{\mathbb{Z}}

\def\H{\mathfrak{H}}

\def\A{\mathscr{A}}
 \def\E{\mathbb{E}}

\def\to{\rightarrow}

\def\bv{\big\vert}

 \def\e{\varepsilon}

\def\H{\mathfrak{H}}

\def\X{\mathfrak{X}}

\def\Var{{\rm Var}}

\newtheorem{thm}{Theorem}[section]
\newtheorem{lemma}[thm]{Lemma}
\newtheorem{cor}[thm]{Corollary}
\newtheorem{prop}[thm]{Proposition}
\newtheorem{obs}[thm]{Remark}
\newtheorem{definition}[thm]{Definition}

\def\1{{\rm l}\hskip -0.21truecm 1}

\begin{document}

\title{Asymptotic behavior of large Gaussian correlated Wishart matrices}

 \date{\today}

\author{Ivan Nourdin \and Guangqu Zheng}

\address{Ivan Nourdin:  Universit\'e du Luxembourg,
Unit\'e de Recherche en Math\'ematiques,
Maison du Nombre,
6 avenue de la Fonte,
L-4364 Esch-sur-Alzette,
Grand Duchy of Luxembourg}.
\email{ivan.nourdin@uni.lu}
\thanks{I. Nourdin was partially supported  by the Grant F1R-MTH-PUL-15CONF at Luxembourg University}

\address{Guangqu Zheng: University of Melbourne, School of Mathematics and Statistics, Peter Hall Building, Monash Road, Parkville  Vic 3010, Australia}
\email{zhengguangqu@gmail.com}


\date{\today}


\begin{abstract}
    
We consider high-dimensional Wishart matrices $d^{-1}\mathcal{X}_{n,d}\mathcal{X}_{n,d}^T$, associated with a rectangular random matrix $\mathcal{X}_{n,d}$ of size $n\times d$ whose entries are jointly Gaussian and correlated. Even if we will consider the case of overall correlation among the entries of $\mathcal{X}_{n,d}$, our main focus is on the case where the  rows of $\mathcal{X}_{n,d}$ are independent copies of a $n$-dimensional stationary centered Gaussian vector of correlation function $s$. 
When $s$ belongs to $\ell^{4/3}(\Z)$, we show that a proper normalization of 
$d^{-1}\mathcal{X}_{n,d}\mathcal{X}_{n,d}^T$ is close in {\it Wasserstein distance} to the corresponding Gaussian ensemble {\it as long as $d$ is much larger than $n^3$}, thus recovering the main finding of \cite{BDER16,JiangLi13} and extending it to a larger class of matrices.
We also investigate the case where $s$ is the correlation function associated with the fractional Brownian noise of parameter $H$. This example is very rich, as it gives rise to a great variety of phenomena with very different natures, depending on how $H$ is located with respect to $1/2$, $5/8$ and $3/4$. 
Notably, when $H>3/4$, our study highlights a new probabilistic object, which we have decided to call the \emph{Rosenblatt-Wishart matrix}.
Our approach crucially relies on the fact that the entries of the Wishart matrices we are dealing with are double Wiener-It\^o integrals, allowing us to make use of multivariate bounds arising from the Malliavin-Stein method and related ideas.
To conclude the paper, we analyze the situation where the
row-independence assumption is relaxed and we also look at the setting of random $p$-tensors ($p\geq 3$), a natural extension of Wishart matrices. 

\end{abstract}

\maketitle

\section{Introduction and main results}

Let $\mathcal{X}_{n,d}=(X_{ij})_{1\leq i\leq n,1\leq j\leq d}$ be a $n\times d$ random matrix whose entries are {\it independent} copies of a real centered  random variable  with unit variance (say).  The (real)  {\it Wishart matrix}  $d^{-1}\mathcal{X}_{n,d}\mathcal{X}_{n,d}^T$, introduced by   the statistician J. Wishart \cite{Wishart28} in the late 20s,  has been very useful in the multivariate statistics and it arises  naturally as the sample covariance matrix. When  $n$ is {\it fixed} and   $d$ goes to infinity, the matrix $d^{-1}  \mathcal{X}_{n,d}\mathcal{X}_{n,d}^T$ converges almost surely to the identity matrix $\mathcal{I}_n$, according to the strong law of large numbers. 
Moreover,  the multivariate central limit theorem implies that 
the fluctuations  of  Wishart matrix around $\mathcal{I}_n$ are Gaussian, provided  $X_{11}$ has the finite fourth moment.

For a long time, the case where $n$ is fixed was enough for applications.
But in the current world filled with large data sets, there has been  a change of paradigm: that both 
$d$ and $n$ are large simultaneously has now become the rule rather than the exception; see \emph{e.g.} Johnstone's ICM survey \cite{Johnstone06}.  In such a context, one can no longer merely rely on the law of large numbers and the classical  central limit theorem to analyze the asymptotic behavior
of the Wishart matrix.

 In order to describe the asymptotic behavior of $d^{-1} \mathcal{X}_{n,d}\mathcal{X}_{n,d}^T$ when {\it both} $d$ and $n$ go to infinity, a classical strategy in random matrix theory consists in   analyzing the weak convergence of its {\it empirical spectral 
distribution} $\mu_{n,d}$, defined as
$\mu_{n,d}=n^{-1} \sum_{i=1}^n \delta_{\lambda_i(n,d)}$, where $\delta_\lambda$ stands for the Dirac mass at $\lambda$ and $\lambda_{1}(n,d)\leq\ldots\leq\lambda_{n}(n,d)$ are the eigenvalues of $d^{-1}  \mathcal{X}_{n,d}\mathcal{X}_{n,d}^T$.
It is   known since Marchenko and Pastur \cite{MP67} that, if $n=n(d)$ {\it diverges} in a way such that $n/d\to c\in(0,\infty)$, then
$\mu_{n,d}$ converges weakly to $\mu_c=\max\{0, 1-c^{-1}\}\delta_0 + ( 2\pi c x)^{-1}\sqrt{(a_+-x)(x-a_-)}{\bf 1}_{[a_-,a_+]}(x)dx$ with  
  $a_{\pm}=(1\pm \sqrt{c})^2$.
At the fluctuation level, it is also known that linear statistics of eigenvalues of Wishart matrix satisfies central limit theorems under certain conditions; see for instance \cite{Chatterjee08} and references therein.

\medskip

However, for some applications the previous way to describe the asymptotic behavior of a large Wishart matrix might not be appropriate (for several possible reasons: because $n$ and $d$ are not of the same order, or because we do not have access to eigenvalues, etc.)
In this work, we take another approach, recently introduced in \cite{BDER16,JiangLi13} and that we shall describe now.
To ease the presentation, we start with the following definition.

\begin{definition}\label{def1}
For each $n\geq 1$, suppose that we have two families  $\{\mathcal{W}_{n,d}:\,d\geq 1\}$, $\{ \mathcal{Z}_{n,d}:\,d\geq 1\}$  of $n\times n$ random matrices. Consider a function $\phi:\mathbb{N}^*\times \mathbb{N}^*\to[0,\infty]$. We say that  $\mathcal{W}_{n,d}$ is \emph{$\phi$-close} to $\mathcal{Z}_{n,d}$ if  the Fortet-Mourier distance $d_{\rm FM}\big(\mathcal{W}_{n,d},  \mathcal{Z}_{n,d} \big)$  between  them tends to zero,  when $d, n\to \infty$ and $\phi(n,d)\to 0$. 
\end{definition}
 
In Definition \ref{def1}, we  use   the \emph{Fortet-Mourier distance} between two random variables with values in  $\mathcal{M}_n(\R)$, the space of $n\times n$ real matrices. Let us recall its definition: if $\mathcal{X}$ and $\mathcal{Y}$ are two such random matrices, then
\begin{equation}\label{dfm}
d_{\rm FM}(\mathcal{X},\mathcal{Y}) := \sup\big\{   \E[g(\mathcal{X})]-\E[g(\mathcal{Y})] \, : \, \|g\|_\infty+\| g \| _{\rm Lip} \leq 1 \, \big\} \,,
\end{equation}
with ${\displaystyle  \qquad\qquad
 \|g\|_\infty := \sup_{A\in\mathcal{M}_n(\R)}|g(A)|\quad\mbox{and}\quad
\|g\|_{\rm Lip} :=\sup_{\substack{A,B\in\mathcal{M}_n(\R)\\A\neq B}}\frac{|g(A)-g(B)|}{\|A-B\|_{\rm HS}}, }
 $

\noindent{}where $\|\cdot\|_{\rm HS}$ is the Hilbert-Schmidt norm on $\mathcal{M}_n(\R)$.      Since $\big(\mathcal{M}_n(\R), \| \cdot \| _{\rm HS}\big)$ is a Polish space, it is well known that $d_{\rm FM}$ characterizes the weak convergence of probability measures on $\mathcal{M}_n(\R)$; see \emph{e.g.} \cite[Section 11.3]{Dudley10}. 
 We will also use the \emph{Wasserstein distance}, which is a stronger distance and defined in a similar way:
 \begin{equation}\label{dwass}
d_{\rm Wass} (\mathcal{X},\mathcal{Y}) = \sup\big\{   \E[g(\mathcal{X})]-\E[g(\mathcal{Y})] \, : \, \| g \| _{\rm Lip} \leq 1 \, \big\} \,.
\end{equation}
It is trivial that $d_{\rm FM}(\mathcal{X},\mathcal{Y})\leq d_{\rm Wass} (\mathcal{X},\mathcal{Y})$, so that any bound on the Wasserstein distance implies the same bound for the Fortet-Mourier distance. 
\medskip

With the above definition and notation in mind, let us go back to the study of high-dimensional fluctuation of  Wishart matrices by considering
a normalized version of $d^{-1}\mathcal{X}_{n,d}\mathcal{X}_{n,d}^T$,  namely
\begin{equation}\label{wnd}
\mathcal{W}_{n,d}=\sqrt{d}  \left( \frac1d\mathcal{X}_{n,d} \mathcal{X}_{n,d}^T  - \mathcal{I}_n\right).
\end{equation}
Also, consider the GOE matrix 
\begin{equation}\label{gn}
\mathcal{Z}_n = (Z_{ij})_{1\leq i,j\leq n},
\end{equation}
where 
$Z_{ii}\sim N(0,2)$, $Z_{ij}\sim N(0,1)$ for $i<j$, $Z_{ij}=Z_{ji}$ for $i>j$, and 
$\{Z_{ij}, i\leq j\}$ are independent. 
 Interestingly, the following phenomenon discovered independently in \cite{BDER16,JiangLi13}  arises:  when the entry distribution of $\mathcal{X}_{n,d}$ is {\it standard Gaussian}, 
 then $\mathcal{W}_{n,d}$ and $\mathcal{Z}_n$ are $\phi$-close (with respect to  the total variation distance\footnote{Due to the explicit density functions of these two random matrix ensembles, the authors of \cite{BDER16} were indeed able to compute the total variation distance as the $L_1$ distance between densities; see also the earlier work \cite{JiangLi13}, in which the same conclusion was derived independently, using quite involved spectral analysis.})  for $\phi(n,d)=n^3/d$.
Otherwise stated, one cannot distinguish
between the laws of  $\mathcal{W}_{n,d}$ and $\mathcal{Z}_n$ when $n$ and $d$ go to infinity with $d$ much larger than $n^3$. 
 It is also proved  in \cite{BDER16} that the condition $n^3=o(d)$ is sharp in the following sense: if we assume this time that $d=o(n^3)$, then the total variation distance
between the laws of  $\mathcal{W}_{n,d}$ and $\mathcal{Z}_n$ goes to $1$; see \cite{BDER16} for precise statements and see also \cite{smooth} for a proof that this phase transition from $d=o(n^3)$ to $n^3=o(d)$ is smooth.   The  work \cite{BDER16} was soon generalized in \cite{BG16} to  the setting that the entries of $\mathcal{X}_{n,d}$ are {\it i.i.d. log-concave}\footnote{that is, $X_{11}$ has a density function $\phi$ such that $\log \phi$ is a concave function.} random variables and some very similar results on the high-dimensional regime were obtained therein.

 \medskip

In the present paper, we take another path of generalization, by  
 relaxing the full independence assumption that is made in \cite{BDER16,BG16} on the entries
of $\mathcal{X}_{n,d}$.  In fact, the most basic phenomena in multivariate analysis is that of correlation -- the tendency of quantities to vary together. And the appearance of correlation usually increases drastically  the complexity of the problem at hand, see \emph{e.g.} \cite[Section 6]{BG16}.  As a first step, we have decided to mainly focus on the case where the entries of $\mathcal{X}_{n,d}$ are {\it Gaussian} and exhibit {\it row independence}, that is, we will allow the columns of $\mathcal{X}_{n,d}$ to be correlated, but not its rows; see Section 4 for further results on  the case of overall correlation.

Here is our precise framework. Let $\H$ be a real separable Hilbert space equipped with the inner product $\langle\cdot,\cdot\rangle_\H$ and the Hilbert norm  $\|\cdot\|_\H$, and let $\{ e_{ij}: i,j\geq 1\}\subset\H$ be a family such that
\begin{equation}\label{covariance}
\langle e_{ij},e_{i'j'}\rangle_\H = \mathbf{1}_{\{i=i'\}}s(j-j').
\end{equation}
In \eqref{covariance},
$s:\mathbb{Z}\to\R$ stands for some correlation function satisfying suitable assumptions to be given later on. One of them is that $s(0)=1$, implying in particular that $\|e_{ij}\|_\H=1$ for all $i,j\geq 1$.

Consider the corresponding Gaussian sequence $X_{ij}=X(e_{ij})\sim N(0,1)$, where $X=\{X(h),\,h\in\H\}$ is a centered Gaussian process indexed by $\H$ such
that $\E[X(g)X(h)]=\langle g,h\rangle_\H$ for all $g,h\in\H$, that is, $X$ is an isonormal process over $\H$.
As before, let $\mathcal{X}_{n,d}^{(s)} $ be the $n\times d$ random matrix given by
\begin{equation}
\label{xnd}
 \mathcal{X}_{n,d}^{(s)}  =  
    \left( \begin{matrix}
    X_{11} & X_{12} &  \ldots   & X_{1d} \\
    X_{21} & X_{22} & \ldots & X_{2d} \\
    \vdots & \vdots & \vdots  &  \vdots  \\
    X_{n1} & X_{n2} &  \ldots   & X_{nd}
\end{matrix}  \right) .  
\end{equation}
Given the form of \eqref{covariance}, we note that the rows $\{ (X_{i1},\ldots,X_{id})$, $i=1,\ldots,n\}$ of $\mathcal{X}_{n,d}^{(s)} $ are independent and identically distributed.

We are now in a position to state our first main result, in which we basically extend the main
result of \cite{BDER16,JiangLi13} to rectangular random matrices of the type (\ref{xnd}), that is, to a situation where the entries of $\mathcal{X}_{n,d}$ are {\it Gaussian} and {\it partially correlated}.

\begin{thm}[Gaussian approximation]\label{thm1} 
Let $\mathcal{X}_{n,d}^{(s)}$ be given by \eqref{xnd}, and consider
\begin{equation}\label{wnds}
\mathcal{W}_{n,d}^{(s)}=(W_{ij})_{1\leq i,j\leq n}=\sqrt{d}  \left( \frac1d\mathcal{X}_{n,d}^{(s)} (\mathcal{X}_{n,d}^{(s)})^T  - \mathcal{I}_n\right).
\end{equation}
Let $\mathcal{G}^{(s)}_{n,d}=(G_{ij})_{1\leq i,j\leq n}$ be a $n\times n$ symmetric random  matrix such that the associated random vector
$$
(G_{11},\ldots, G_{1n},G_{21},\ldots, G_{2n},\ldots,G_{n1},\ldots,G_{nn})^T
$$ 
of $\R^{n^2}$ is Gaussian and 
has the same covariance matrix as 
$$
(W_{11},\ldots, W_{1n},W_{21},\ldots, W_{2n},\ldots,W_{n1},\ldots,W_{nn})^T.
$$ 
Then, for any $n,d\geq 1$, one has
\begin{eqnarray}
d_{\rm Wass} \big(  \mathcal{W}_{n,d}^{(s)},  \mathcal{G}_{n,d}^{(s)}   \big)
&\leq& \sqrt{\frac{192n^3}{ \sum_{|k|\leq d} (1 - \frac{|k|}{d} ) s(k)^2  }\times \frac{1}d\left(  \sum_{\vert k \vert \leq d} \vert s(k) \vert^{4/3}   \right)^{3}}\label{cl1}\\
&\leq& \sqrt{\frac{192n^3}d\left(  \sum_{\vert k \vert \leq d} \vert s(k) \vert^{4/3}   \right)^{3}}\notag \quad (\mbox{\rm since ${\displaystyle\sum_{|k|\leq d} (1 - \frac{|k|}{d} ) s(k)^2\geq s(0)^2=1}$}).
\end{eqnarray}
Moreover, if $s\in \ell^2(\mathbb{Z})$ and if we let $\mathcal{Z}^{(s)}_{n}=(Z_{ij})_{1\leq i,j\leq n}$ denote a $n\times n$ symmetric random matrix such that
$Z_{ii}\sim N(0,2\|s\|^2_{\ell^2(\Z)})$, $Z_{ij}\sim N(0,\|s\|^2_{\ell^2(\Z)})$ for $i<j$, $Z_{ij}=Z_{ji}$ for $i>j$, and 
$\{Z_{ij}, 1\leq i\leq j\leq n\}$ independent, we have
\begin{equation}
d_{\rm Wass} \big(  \mathcal{G}_{n,d}^{(s)}, \mathcal{Z}_{n}^{(s)}    \big)
\leq \frac{2\sqrt{n(n+1)}}{\|s\|_{\ell^2(\Z)}} \left(
\sum_{|k|>d}s(k)^2+\frac1d\sum_{|k|\leq d}|k| s(k)^2\right) . \label{cl2}
\end{equation}
\end{thm}
       
\medskip
        
       \begin{obs} {\rm
       (i) For instance, if $s(r)={\bf 1}_{\{r=0\}}$, that is,  if the entries $\{X_{ij}: i,j\geq 1\}$ of $\mathcal{X}_{n,d}$ are i.i.d standard Gaussian, then  we recover from  Theorem \ref{thm1}   the result of \cite{BDER16,JiangLi13}; $\mathcal{W}_{n,d}$ given by (\ref{wnd}) is close to $\mathcal{Z}_{n}$ given by (\ref{gn}) when $n^3/d\to 0$. 

\noindent{\qquad (ii)}  
If we assume that $s\in\ell^{4/3}(\Z)$, then (\ref{cl1}) leads to $d_{\rm Wass} \big(  \mathcal{W}_{n,d}^{(s)},  \mathcal{G}_{n,d}^{(s)}   \big)  =O\big(\sqrt{n^3/d}\big)$.  In this case, $ \mathcal{W}^{(s)}_{n,d}$ continues to be close to  $\mathcal{G}^{(s)}_{n,d}$ as soon as  $n^3/d\to 0$, exactly like in the full independent case considered in \cite{BDER16,BG16,JiangLi13}; see also (i).
}
\end{obs}

As we just pointed out in the previous remark, $\mathcal{W}_{n,d}^{(s)}$ and $\mathcal{G}_{n,d}^{(s)}$ are asymptotically close as long as $s\in\ell^{\frac43}(\Z)$ and 
$n^3/d\to 0$. What happens when $s\in\ell^2(\mathbb{Z})\setminus\ell^{\frac43}(\Z)$ or when $s\not\in\ell^2(\mathbb{Z})$? And how close are $\mathcal{G}_{n,d}^{(s)}$ and  $\mathcal{Z}_{n}^{(s)} $?
To exhibit an interesting situation where different behaviors may arise, starting from now on we will focus on the case where $s$ is the correlation function of the fractional Brownian  noise of Hurst index $H\in(0,1)$, that is,
\begin{equation}\label{frac}
s(k)=s_H(k)=\frac{1}{2}\big(|k+1|^{2H}+|k-1|^{2H}-2|k|^{2H}\big),\quad k\in\mathbb{Z}.
\end{equation}
When $H\neq \frac12$, one has $\vert  s_H(k) \vert  \sim c \vert k \vert^{2H-2}$ as $\vert k \vert\to\infty$ (with $c>0$ a constant whose value is immaterial and may change from one instance to another). It is a straightforward exercise to check that $s_H\in\ell^2(\Z)$ if and only if $H\in(0,\frac34)$. Moreover, as $d\to\infty$,
\begin{eqnarray*}
\frac{1}{d}  \left(
\sum_{|k|\leq d} |s_H(k)|^{4/3} 
 \right)^{3}&\sim& c\left\{
\begin{array}{ll}
1/d&\quad\mbox{if $0<H<5/8$}\\
(\log d)^{3}/d&\quad \mbox{if $H=5/8$}\\
d^{8H-6}&\quad \mbox{if $5/8<H<1$}
\end{array}
\right. ;\\
\sum_{|k|>d}s_H(k)^2&\sim& c \,d^{4H-3}\quad\mbox{for }H\in(0, 1/2)\cup(1/2,3/4);\\
\sum_{|k|\leq d} (1 - \frac{|k|}{d} ) s_H(k)^2&\sim& c\,\log d \quad\mbox{for }H=3/4 ;\\
\frac{1}{d}\sum_{|k|\leq d}|k|s_H(k)^2&\sim& c\left\{
\begin{array}{ll}
1/d&\quad\mbox{if $0<H< 1/2$}\\
d^{4H-3}&\quad \mbox{if $1/2<H<1$}
\end{array}
\right. .
\end{eqnarray*}
As a result, for $s_H$ given by \eqref{frac}, we deduce from (\ref{cl1}) that
$\mathcal{W}^{(s_H)}_{n,d}$ is $\phi$-close to $\mathcal{G}^{(s)}_{n,d}$ when $H\in(0,\frac34]$ and
\begin{equation}\label{limit}
\phi(n,d)=\frac{n^3}d\,{\bf 1}_{\{0<H<5/8\}}
+
n^3\frac{\log^3 d}{d}\,{\bf 1}_{\{H=5/8\}}
+
n^3 d^{8H-6}\,{\bf 1}_{\{5/8<H<3/4\}}+
\frac{n^3}{\log d}\,{\bf 1}_{\{H=3/4\}}.
\end{equation}
On the other hand, when $H\in(0,\frac34)$ (otherwise $\mathcal{Z}^{(s)}_{n}$ is not defined), we deduce from (\ref{cl2})
 that
$\mathcal{G}^{(s)}_{n,d}$ is $\psi$-close to $\mathcal{Z}^{(s)}_{n}$
for
\begin{equation}\label{limit2}
\psi(n,d)=\frac{n}d\,{\bf 1}_{\{0<H<1/2\}}
+
n d^{4H-3}\,{\bf 1}_{\{1/2<H<3/4\}}.
\end{equation}

\medskip

We summarize the above discussion as follows.

\begin{cor}[High-dimensional regime for fractional noise entries and Gaussian approximation] \label{cor-fbm} Assume that $s = s_H$ is given  by \eqref{frac} with $H\in(0, \frac34]$. 
Then 
\begin{itemize}

\item[(i)] when $H\in(0, 1/2)$: $\mathcal{W}^{(s_H)}_{n,d}$ is $\phi$-close to $\mathcal{G}^{(s_H)}_{n,d}$ for $\phi(n,d)=n^3/d$, whereas
$\mathcal{G}^{(s_H)}_{n,d}$ is $\psi$-close to $\mathcal{Z}^{(s_H)}_{n}$ for $\psi(n,d)=n/d$;

\item[(ii)] when $H=1/2$: $\mathcal{W}^{(s_H)}_{n,d}$ is $\phi$-close to $\mathcal{G}^{(s_H)}_{n,d}$ for $\phi(n,d)=n^3/d$, whereas
$\mathcal{G}^{(s_H)}_{n,d}$ and $\mathcal{Z}^{(s_H)}_{n}$ have the same law;

\item[(iii)] when $H\in(1/2,5/8)$: $\mathcal{W}^{(s_H)}_{n,d}$ is $\phi$-close to $\mathcal{G}^{(s_H)}_{n,d}$ for $\phi(n,d)=n^3/d$, whereas
$\mathcal{G}^{(s_H)}_{n,d}$ is $\psi$-close to $\mathcal{Z}^{(s_H)}_{n}$ for $\psi(n,d)=nd^{4H-3}$;

\item[(iv)] for $H=5/8$: $\mathcal{W}^{(s_H)}_{n,d}$ is $\phi$-close to $\mathcal{G}^{(s_H)}_{n,d}$ for $\phi(n,d)=(n^3\log^3 d)/d$, whereas
$\mathcal{G}^{(s_H)}_{n,d}$ is $\psi$-close to $\mathcal{Z}^{(s_H)}_{n}$ for $\psi(n,d)=n/\sqrt{d}$;

\item[(v)] for $H\in(5/8, 3/4)$:  $\mathcal{W}^{(s_H)}_{n,d}$ is $\phi$-close to $\mathcal{G}^{(s_H)}_{n,d}$ for $\phi(n,d)=n^3d^{8H-6} $, whereas
$\mathcal{G}^{(s_H)}_{n,d}$ is $\psi$-close to $\mathcal{Z}^{(s_H)}_{n}$ for $\psi(n,d)=nd^{4H-3}$;

\item[(vi)] for $H=3/4$:  $\mathcal{W}^{(s_H)}_{n,d}$ is $\phi$-close to $\mathcal{G}^{(s_H)}_{n,d}$ for $\phi(n,d)=n^3(\log d)^{-1} $,  whereas
$\mathcal{Z}^{(s_H)}_{n}$ is not defined.

\end{itemize}
\end{cor}

\begin{obs}\label{rem1.2} 
{\rm
(1) Note that the particular case where $H=1/2$ reduces to the case of full independence and gives us the same high-dimensional regime as in \cite{BDER16,JiangLi13}.   One of the main ingredients in the proof of Theorem \ref{thm1} is the Stein's method developed in the work \cite{NPRev10}, in which the Malliavin calculus was coupled with the multivariate Stein's method
in order to deal with the Wasserstein distance.
This explains why, in the present paper, our bounds are for the Wasserstein distance and not the total variation distance $d_{\rm TV}$ as in \cite{BDER16,BG16,JiangLi13};   see also Section 5.

\noindent{\qquad (2)} When $H=5/8$, the high-dimensional regime looks similar to that derived in \cite{BG16}, that is, assuming the entries $X_{ij}$ are i.i.d log-concave and $( n^3  \log^2 d  )/d\to 0$, one has 
$
  d_{\rm TV}\big( \mathcal{W}_{n,d}, \mathcal{G}_n \big) \to 0$,
where $ \mathcal{G}_n$ belongs to the GOE. The log-terms in both regimes seem to be the price paid for deviating away from being Gaussian or independent,  see also the critical case $H=3/4$.

 }

 \end{obs}

In the next  result, we finally explain what happens in the case $H\in(3/4,1)$. In this case,  an interesting and new phenomenon appears: properly scaled, the Wishart matrix associated with $\mathcal{X}_{n,d}$ given by (\ref{xnd}) and $s=s_H$ given by (\ref{frac}) converges to the so-called \emph{Rosenblatt-Wishart matrix} for a suitable range of $n$ and $d$.

\begin{thm}[Rosenblatt approximation]\label{thm2}
Consider $s_H$ given by \eqref{frac} with $H\in( 3/4,1)$, set
\begin{equation}\label{wnd2}
\wh{\mathcal{W}}^{(s_H)}_{n,d} = d^{2-2H}  \left( \frac1d\mathcal{X}_{n,d}^{(s_H)} (\mathcal{X}_{n,d}^{(s_H)})^T  - \mathcal{I}_n\right) 
\end{equation}
and let $\mathcal{R}^{(H)}_{n}$ be the $n\times n$  Rosenblatt-Wishart matrix with Hurst parameter $H$ {\rm (\emph{see} Definition \ref{defi-rose})}.     Then, there exists a finite constant $c_H>0$ depending only on $H$ such that,
for any $n,d\geq 1$,
\[
 d_{\rm Wass}\big( \wh{\mathcal{W}}^{(s_H)}_{n,d} ,  \mathcal{R}^{(H)}_{n}   \big)
 \leq c_H\,n\,d^{(3-4H)/2}.
\]
In other words,  the scaled Wishart matrix 
$\wh{\mathcal{W}}_{n,d}^{(s_H)}$ is $\phi$-close to  the Rosenblatt-Wishart matrix $\mathcal{R}^{(H)}_{n}$ for $\phi(n,d)=n^2\,d^{3-4H}$.

\end{thm}

The above theorem addresses the non-central limit theorems in the context of large random matrices and a crucial step in its proof is to construct explicitly a coupling of  $\wh{\mathcal{W}}^{(s_H)}_{n,d}$ and  $\mathcal{R}^{(H)}_{n}  $ using the \emph{self-similarity} of fractional Brownian motion, with which we bound the Wasserstein distance by the $L^2$-distance.

\medskip

The rest of the paper is organized as follows.
In Section 2, we develop all the material needed for the proof of Theorem \ref{thm1}, and we give its proof in the end.
Section 3 is devoted to the proof of Theorem \ref{thm2} as well as the introduction to the new notion of Rosenblatt-Wishart matrix.
In Section 4, we analyze the situation where the row-independence assumption is relaxed and we also look at the setting of random $p$-tensors ($p\geq 3$), a natural extension of Wishart matrices.
Finally, we propose some related open problems for future research in Section 5. \\

\noindent{\bf Acknowledgement.}  We  thank St\'ephane Chr\'etien  for mentioning  the paper \cite{BG16} to one of us (Ivan), which initiated  our investigation.

\section{Gaussian approximation}

The basic tools we use  throughout this work is the Malliavin calculus and  Stein's method, whose combination is commonly known as the Malliavin-Stein approach. Such an approach was motivated to quantify   Nualart and Peccati's fourth moment theorem \cite{FMT}, and it has been extensively developed  by the authors of  \cite{NP09} as well as their collaborators; see \cite{bluebook} for a comprehensive treatment. This Malliavin-Stein approach has turned out to be very applicable in quantifying limit theorems on a Gaussian space. More specifically, we are going to use its multidimensional  version  derived in \cite{NPRev10} to investigate the high-dimensional regime concerning the Gaussian approximation of Wishart matrices.   In Section 2.1, we collect   several basic facts. Section 2.2 is devoted to the proof of Theorem \ref{thm1}. We refer the readers to the monograph \cite{bluebook} for any unexplained notation. 

\subsection{Preparation of the proof of Theorem \ref{thm1}}

Let us first recall the framework put in the introduction: $X=\{X(h),\,h\in \H\}$ is an isonormal process over a real separable Hilbert space $\H$, defined on some probability space $(\Omega,\mathcal{F},\mathbb{P})$.

For every $p\geq 1$, we let $\mathcal{H}_p$ denote the $p$th Wiener chaos of $X$, that is, the
closed linear subspace of $L^2(\Omega)$ generated by the random variables of the form 
$\{H_p(X(h)),\,h\in\H,\,\|h\|_\H=1\}$, where $H_p$ stands for the $p$th Hermite polynomial\footnote{$H_1(x)=x$, $H_2(x)=x^2-1$, $H_3(x)=x^3-3x$ and $H_{p+1}(x) = xH_p(x) - p H_{p-1}(x)$ for every $p\geq 2$.}. The relation that $I_p(h^{\otimes p})= H_p(X(h))$ for unit vector $h\in\H$ can be extended to a linear isometry between the symmetric $p$th tensor product $\H^{\odot p}$ (equipped with the modified norm $\sqrt{p!}\|\cdot\|_{\H^{\otimes p}}$)
and the $p$th Wiener chaos $\mathcal{H}_p$.

Suppose $(h_i, i\geq 1)$ is an orthonormal basis of $\H$, and consider $f\in\H^{\otimes p}$ and $g\in\H^{\otimes q}$ with $p,q\geq 1$. With $ f(i_1,\ldots, i_p) = \langle f,  h_{i_1}\otimes\cdots \otimes h_{i_p} \rangle_{\H^{\otimes p}}$ and $g(i_1,\ldots, i_q) = \langle g,  h_{i_1}\otimes\cdots \otimes h_{i_q} \rangle_{\H^{\otimes q}}$, 
we can express them as
\begin{align}
f =  \sum_{i_1, \ldots, i_p=1}^\infty f(i_1,\ldots, i_p) h_{i_1}\otimes\cdots \otimes h_{i_p}  ~~ \text{and} ~ ~g=  \sum_{i_1, \ldots, i_q=1}^\infty g(i_1,\ldots, i_q) h_{i_1}\otimes\cdots \otimes h_{i_q}  . \label{f-exp}
\end{align}
For $r\in\{1,\ldots, p\wedge q\}$, the $r$-contraction of $f$ and $g$ is the element in $\H^{\otimes p+q-2r}$ defined by 
\begin{align*}
f\otimes_r g =  \sum_{\substack{ i_1, \ldots, i_{p-r}\geq 1  \\ j_1, \ldots, j_{q-r}\geq 1 }}  (f\star_rg)(i_1,\ldots, i_{p-r},j_1,\ldots, j_{q-r} )    h_{i_1}\otimes \cdots \otimes h_{i_{p-r}} \otimes h_{j_1}\otimes \cdots \otimes h_{j_{q-r}}
\end{align*}
where $$(f\star_rg)(i_1,\ldots, i_{p-r},j_1,\ldots, j_{q-r} ) = {\displaystyle \sum_{k_1, \ldots, k_r=1}^\infty  f(k_1,\ldots, k_r, i_1, \ldots,  i_{p-r})    g(k_1,\ldots, k_r, j_1, \ldots,  j_{q-r}) }.$$  
Contractions naturally appear in the product formula for multiple Wiener-It\^o integrals: for $f\in\H^{\odot p}$ and $g\in\H^{\odot q}$ with $p,q\geq 1$, it holds that
\begin{align}\label{form-prod}
 I_p(f) I_q(g) = \sum_{r=0}^{p\wedge q} r! {p\choose r} {q\choose r} I_{p+q-2r}\big( f \wt{\otimes}_r g \big) \,,
\end{align}
where $ f \wt{\otimes}_r g $ stands for the symmetrization of  $f \otimes_r g$; see \emph{e.g.} \cite[Theorem 2.7.10]{bluebook}. 

We will also need the notion of Malliavin derivative $D$ with respect to $X$ but only its action on a fixed Wiener chaos: for $f\in\H^{\odot p}$ of the form \eqref{f-exp}, the Malliavin derivative of $I_p(f)$ is the random  element of $\H$ given by   
\[
DI_p(f) = p \sum_{i=1}^\infty  I_{p-1}(f\otimes_1 h_i) h_i = p \sum_{i_1, \ldots, i_p\geq 1} f(i_1,\ldots, i_p) I_{p-1}(h_{i_2}\otimes\cdots \otimes h_{i_p} )   h_{i_1} \, .
\]

Bearing all this in mind, let us go back to the rectangular matrix $\mathcal{X}_{n,d}$  defined by (\ref{xnd}).
Since $I_1(h)=X(h)$ for all $h\in\H$,
the entries of $\mathcal{X}_{n,d}$ are realized as elements in the first Wiener chaos $\mathcal{H}_1$. As a consequence, due  to either the very definition of $\mathcal{H}_2$ (when $i=j$) or the product formula (when $i\neq j$), the $(i,j)$th entry $W_{ij}$ of $\mathcal{W}^{(s)}_{n,d}$ given by (\ref{wnds}) belongs to the second Wiener chaos $\mathcal{H}_2$: more precisely,
\begin{align}
W_{ij}=\left\{
\begin{array}{cc}
{\displaystyle \frac{1}{\sqrt{d}} \sum_{k=1}^d \big(X_{ik}^2 - 1\big) }&\,\mbox{if $i=j$}\\
\quad\\
{\displaystyle \frac{1}{\sqrt{d}} \sum_{k=1}^d X_{ik}X_{jk} }&\,\mbox{if $i\neq j$}\\
\end{array}
\right\}
= I_2(f_{ij}^{(d)})\,, \label{Wndii}
\end{align}
with the kernel   $$f_{ij}^{(d)} =  \frac{1}{2\sqrt{d}} \sum_{k=1}^d ( e_{ik}\otimes e_{jk} + e_{jk}\otimes e_{ik}).$$

Before we present the proof of Theorem \ref{thm1}, we prepare several important facts on double Wiener-It\^o integrals.

\medskip

{\bf\small Fact 1.} For any $f, g\in\H^{\odot 2}$, on has  $$\langle D I_2(f), D I_2(g) \rangle_\H - \E\big[ \langle D I_2(f), D I_2(g) \rangle_\H\big] = 4 I_2\big( f\wt{\otimes}_1 g \big).$$
This can be verified by using the product formula \eqref{form-prod}.

\medskip

{\bf\small Fact 2.} For kernels $f_{ij}^{(d)}$ given in \eqref{Wndii}, we have $f_{ij}^{(d)} \otimes_1 f_{kl}^{(d)} = 0$, whenever $\{ i,j\} \cap \{ k,l\} =\emptyset$. Here we may abuse the notation $\{ i,j\} = \{i\}$ if $i=j$.  This fact follows from the specific shape of \eqref{covariance}.

\medskip

{\bf\small Fact 3.} For kernels $f_{ij}^{(d)}$ given in \eqref{Wndii},  we can  obtain by following the same computations as in \cite[Page 134-135]{bluebook} that
 \begin{eqnarray*}
 \| f_{ii}^{(d)} \otimes_1 f_{ii}^{(d)} \|^2_{\H^{\otimes 2}} & =&  \left\| \frac{1}{d} \sum_{k,\ell=1}^d \big(e_{ik}\otimes e_{i\ell}\big) s(k-\ell) \right\|^2_{\H^{\otimes 2}}   \\
 & =&  \frac{1}{d^2} \sum_{k,\ell, u, v=1}^d s(k-\ell)s(\ell-u) s(u-v)s(v-k)  
 \leq  \frac{1}{d} \left( \sum_{\vert k \vert \leq d} \vert s(k) \vert^{4/3} \right)^{3} 
 \end{eqnarray*}
  whereas, for $i\neq j$,
    \begin{align}
    \big\| f_{ij}^{(d)} \otimes_1 f_{ij}^{(d)} \big\| ^2_{\H^{\otimes 2}} &=  \left\| \frac{1}{4d} \sum_{k,\ell=1}^d \big( e_{ik}\otimes e_{i\ell} + e_{jk}\otimes e_{j\ell}   \big) s(k-\ell) \right\|^2_{\H^{\otimes 2}}    \notag\\
    & \leq \frac{1}{8} \big\| f_{ii}^{(d)} \otimes_1 f_{ii}^{(d)} \big\|^2_{\H^{\otimes 2}} +  \frac{1}{8} \big\| f_{jj}^{(d)} \otimes_1 f_{jj}^{(d)} \big\|^2_{\H^{\otimes 2}}  \notag \\
  & \leq   \frac{1}{4d} \left( \sum_{\vert k \vert \leq d} \vert s(k) \vert^{4/3} \right)^{3}  \,.    \label{doubledol}
  \end{align}
Moreover, for any $i,j,k,l$, we have 
\[
\big\| f_{ij}^{(d)} \otimes_1  f_{kl}^{(d)} \big\| ^2_{\H^{\otimes 2}}  
= \big\langle  f_{ij}^{(d)} \otimes_1  f_{ij}^{(d)},  f_{kl}^{(d)} \otimes_1  f_{kl}^{(d)} \big\rangle_{\H^{\otimes 2}}  \leq  \frac{1}{d} \Big( \sum_{\vert k \vert \leq d} \vert s(k) \vert^{4/3} \Big)^{3}   \,,
\] 
where the  equality above follows from the definition of contractions.

\medskip

{\bf\small Fact 4.} Finally, we state the main ingredient for our proof and   we will only use it with $p_1=\ldots=p_m=1$ (in which case it provides a bound for the Wasserstein distance between two $m$-dimensional Gaussian vectors) and $p_1=\ldots=p_m=2$.

\medskip

\begin{prop}[see Corollary 3.6 in \cite{NPRev10}]  \label{NPR10-thm} Fix integers $m\geq 2$ and $1\leq p_1\leq \ldots \leq p_m$. Consider a vector $F = ( F_1, \ldots, F_m) = \big( I_{p_1}(f_1), \ldots,  I_{p_m}(f_m) \big)$ with $f_j\in\H^{\odot p_j}$ 
for each $j$. On the other hand, let $C$ be an invertible covariance matrix, and let $Z\sim N_m(0, C)$. Then
\[
d_{\rm Wass} ( F, Z ) \leq \| C^{-1} \| _{\rm op} \| C \| _{\rm op}^{1/2} \left(  \sum_{1\leq i, j\leq m} \E\Big[ \big( C_{ij}-p_j^{-1} \langle DF_i, DF_j \rangle_\H    \big)^2\Big] \right)^{1/2} \,,
\]
where $\| \cdot \| _{\rm op}$ denotes the usual operator norm. 
\end{prop}

\subsection{Proof of Theorem \ref{thm1}}

We are now ready to give the proof of Theorem \ref{thm1}.
It is divided into several steps.

\medskip

\underline{\it Step 1}: {\it passing from symmetric matrices to vectors}.
Since the entries of $\mathcal{W}_{n,d}^{(s)}$ are double Wiener-It\^o integrals, we would like to apply Proposition \ref{NPR10-thm}. But Proposition \ref{NPR10-thm} is stated for vectors, not for matrices. So, as a  first step, we need to explain how we can reduce to vectors.
If $\mathcal{Z}=(Z_{ij})_{1\leq i,j\leq n}$ is a $n\times n$ random {\it symmetric} matrix,
the notation $\mathcal{Z}^{\rm half} $ indicates the $n(n+1)/2$-dimensional random vector formed by the upper-triangular entries, namely:
\begin{equation}\label{half}
\mathcal{Z}^{\rm half} = ( Z_{11}, Z_{12}, \ldots, Z_{1n}, Z_{22}, Z_{23}, \ldots, Z_{2n},\ldots, Z_{nn}  )^T.
\end{equation}

 \medskip

\begin{lemma}\label{22}
Let $\mathcal{X}$ and $\mathcal{Y}$ be two \emph{symmetric} random matrices of $\mathcal{M}_n(\R)$.
Then
$$
d_{\rm Wass}(\mathcal{X},\mathcal{Y})\leq \sqrt{2}\,d_{\rm Wass}(\mathcal{X}^{\rm half},\mathcal{Y}^{\rm half}).
$$
Here $d_{\rm Wass}(\mathcal{X},\mathcal{Y})$ is defined according to \eqref{dwass}, whereas $d_{\rm Wass}(\mathcal{X}^{\rm half},\mathcal{Y}^{\rm half})$ stands for the Wasserstein distance between random variables with values  in $\R^{n(n+1)/2}$, that is,
$$
d_{\rm Wass}(\mathcal{X}^{\rm half},\mathcal{Y}^{\rm half})=\sup\big\{  \E[g(\mathcal{X}^{\rm half})] -  \E[g(\mathcal{Y}^{\rm half})] \,:\, \| g\| _{\rm Lip} \leq 1 \,\big\} \, ,
$$
where the $\| g\| _{\rm Lip} $ stands for the usual Lipschitz constant of  a function $g:\R^{n(n+1)/2}\to\R$.
\end{lemma}

\noindent
{\it Proof}. If $x\in\R^{n(n+1)/2}$, we define $M_x$ to be the $n\times n$ symmetric  matrix such that
$M_x^{\rm half}=x$.
Let $g:\mathcal{M}_n(\R)\to\R$ be $1$-Lipschitz with respect to the Hilbert-Schmidt norm.
We have
$$
\big|\E[g(\mathcal{X})]-\E[g(\mathcal{Y})]\big| = \sqrt{2}~\big| \E[\widetilde{g}(\mathcal{X}^{\rm half})]-
 \E[\widetilde{g}(\mathcal{Y}^{\rm half})]\big|,
$$
where $\widetilde{g}:\R^{n(n+1)/2}  \to\R$ is defined by $\widetilde{g}(x)=\frac{1}{\sqrt{2}}g(M_x)$. Since 
$$
\big|\widetilde{g}(x)-\widetilde{g}(y)\big|=\frac{1}{\sqrt{2}} ~ \big|g(M_x)-g(M_y)\big|
\leq\frac{1}{\sqrt{2}}\|M_x-M_y\|_{\rm HS}\leq \|x-y\|,
$$
we deduce that $
\big|\E[g(\mathcal{X})]-\E[g(\mathcal{Y})]\big| \leq \sqrt{2}\,d_{\rm Wass}(\mathcal{X}^{\rm half},\mathcal{Y}^{\rm half})$,
thus concluding the proof by taking the supremum over $g$.
\qed

\medskip

\underline{\it Step 2}: {\it estimating the operator norm}. Let us now look at the common covariance matrix $C$   of $(\mathcal{W}_{n,d}^{(s)})^{\rm half}$ and $(\mathcal{G}^{(s)}_{n,d})^{\rm half}$.
It is diagonal with entries given by   
\begin{align}\label{diags}
  \E\big[  W_{ii}^2 \big]   =  \frac{2}{d} \sum_{k,\ell=1}^d s(k-\ell)^2 \,\,\, \text{for each $i$}
\,\,\,  \text{and}   \,\,\,
  \E\big[  W_{ij}^2 \big]   = \frac{1}{d} \sum_{k,\ell=1}^d s(k-\ell)^2  \,\,\, \text{for  $i < j$.}
\end{align}
It follows immediately that 
$$\| C \|_{\rm op}^{1/2} \| C^{-1} \|_{\rm op}=   \sqrt{\frac{2d}{ \sum_{k,\ell=1}^d s(k-\ell)^2  }} =\sqrt{\frac{2}{ \sum_{|j|\leq d} (1 - \frac{|j|}{d} ) s(j)^2  }}.
$$ 

\medskip

\underline{\it Step 3}: {\it estimating the variance of $\langle DW_{ij}, DW_{k\ell} \rangle_\H$}.
The entries $W_{ij} = I_2(f_{ij}^{(d)})$ being elements of second Wiener chaos, see (\ref{Wndii}), by {\bf\small Fact 1} and isometry relation for multiple integrals we have
\begin{equation}\label{nonzero}
\Var\Big(  \frac{1}{2} \langle DW_{ij}, DW_{k\ell} \rangle_\H \Big) =   8 \big\| f_{ij}^{(d)} \otimes_1 f_{k\ell}^{(d)} \big\| ^2_{\H^{\otimes 2}} \leq \frac{8}{d} \left( \sum_{\vert k \vert \leq d} \vert s(k) \vert^{4/3} \right)^{3}  \,,
\end{equation}
where the last inequality follows from {\bf\small Fact 3}. Moreover, if $\{ i,j\} \cap \{ k,\ell\} =\emptyset$,  {\bf\small Fact 2} implies
\begin{equation}\label{zero}
\Var\Big(  \frac{1}{2} \langle DW_{ij}, DW_{k\ell} \rangle_\H \Big)=0.
\end{equation}

\underline{\it Step 4}: {\it proving \eqref{cl1}}.
Proposition \ref{NPR10-thm} (with $m=n(n+1)/2$ and $p_1=\ldots=p_m=2$) together with the conclusion of Step 2 give us the following bound
$$
d_{\rm Wass}\big( (\mathcal{W}^{(s)}_{n,d})^{\rm half}   ,  (\mathcal{G}^{(s)}_{n,d})^{\rm half}  \big) \leq \sqrt{\frac{2}{ \sum_{|j|\leq d} (1 - \frac{|j|}{d} ) s(j)^2  }}\left(  \sum_{\substack{1\leq i \leq j \leq n \\  1\leq k \leq \ell \leq n}}  \Var\Big( \frac{1}{2} \langle DW_{ij}, DW_{k\ell} \rangle_\H \Big)    \right)^{1/2}.
$$
We deduce from (\ref{zero}) that the sum $\sum_{\substack{1\leq i \leq j \leq n \\  1\leq k \leq \ell \leq n}} $ in the previous inequality can be replaced by 
$ \sum_{(i,j,k,\ell)\in\mathcal{I}} 
$,
where  the set $\mathcal{I}: = \big\{ (i,j,k,\ell)\in\{1,\ldots, n\}^4:$ $i,j,k,\ell$ are not mutually distinct$\big\}$ has the cardinality $n^4-n(n-1)(n-2)(n-3)$, a quantity bounded by $6n^3$.
Considering also the bound (\ref{nonzero}), we finally get 
$$
d_{\rm Wass}\big( (\mathcal{W}^{(s)}_{n,d})^{\rm half}   ,  (\mathcal{G}^{(s)}_{n,d})^{\rm half}  \big) \leq \sqrt{\frac{96}{ \sum_{|j|\leq d} (1 - \frac{|j|}{d} ) s(j)^2  }\times \frac{n^3}d \left(\sum_{\vert k \vert \leq d} \vert s(k) \vert^{4/3} \right)^{3}  }.
$$
which, thanks to Lemma \ref{22}, gives exactly (\ref{cl1}). 

\bigskip

\underline{\it Step 5}: {\it proving \eqref{cl2}}. If $C$ denotes this time the covariance matrix
of $(\mathcal{Z}_n^{(s)})^{\rm half}$ (that is, $C$ is diagonal with diagonal entries either equal to
$2\|s\|^2_{\ell^2(\Z)}$ or $\|s\|^2_{\ell^2(\Z)}$), we have
$\|C\|_{\rm op}^{1/2}= \sqrt{2}\|s\|_{\ell^2(\Z)}$ and $\|C^{-1}\|_{\rm op}= \|s\|_{\ell^2(\Z)}^{-2}$.
We deduce, according to Proposition \ref{NPR10-thm} with $m=n(n+1)/2$ and $p_1=\ldots=p_m=1$, that
\begin{eqnarray*}
d_{\rm Wass}\big(  (\mathcal{G}^{(s)}_{n,d})^{\rm half}, (\mathcal{Z}^{(s)}_{n})^{\rm half}    \big)
\leq 
\frac{\sqrt{2n(n+1)}}{\|s\|_{\ell^2(\Z)}}  \left(  \sum_{|j|>d}s(j)^2+\frac1d\sum_{\vert j\vert \leq d}|j|s(j)^2 \right)\,.
\end{eqnarray*}
Relying on Lemma \ref{22} again, we obtain (\ref{cl2}).  
\qed

 \section{Rosenblatt approximation}

\subsection{Preliminaries on fractional Brownian motion}\label{sec-fBm}
We consider in this section a $n$-dimensional fractional Brownian motion with Hurst parameter $H\in (0,1)$, that is, a centered Gaussian process 
$
B=\{B_t=(B^1_t,\ldots,B^n_t);\,t\in\R_+\}
$,
where $B^1,\ldots,B^n$ are $n$ independent copies of a real fractional Brownian motion  with
covariance function
$
R_H(s,t)=\frac{1}{2}\big(s^{2H}+t^{2H}-|t-s|^{2H}\big)
$,
for $i=1,\ldots,n$.

The two following fundamental properties of the fractional Brownian motion 
will be used throughout the sequel: 
\begin{itemize}
\item 
it is $H$-{\it selfsimilar}, that is, $(B_{ct})_{t\geq 0}\overset{\rm law}{=}c^H(B_t)_{t\geq 0}$ for all $c>0$; 
\item it has {\it stationary increments}, that is,
$(B_{t+h}-B_h)_{t\geq 0}\overset{\rm law}{=}(B_t)_{t\geq 0}$ for all $h>0$.
\end{itemize}
We will also need a few facts about its Gaussian structure:
let $\mathcal{E}_n$ be the set of step-functions on $\R_+$ with values in $\R^n$ and 
consider the Hilbert space $\H_n$ defined as the closure of $\mathcal{E}_n$ with respect to the scalar product induced by
\begin{equation}\label{scalar-fbm}
\big\langle ({\bf 1}_{[0,s_1]},\ldots,{\bf 1}_{[0,s_n]}), 
({\bf 1}_{[0,t_1]},\ldots,{\bf 1}_{[0,t_n]})
\big\rangle_{\H_n}=\sum_{i=1}^n R_H(s_i,t_i).
\end{equation}
Then the mapping
$
({\bf 1}_{[0,t_1]},\ldots,{\bf 1}_{[0,t_n]})\in \mathcal{E}_n \mapsto \sum_{i=1}^n B^i_{t_i}  
$
can be extended to an isometry between $\H_n$ and the Gaussian 
space associated with $B=(B^1,\ldots,B^n)$. 
We denote this isometry by $\varphi\mapsto B(\varphi)$ and
the process $\{B(\varphi):\,\varphi\in\H_n\}$ is an isonormal Gaussian process by construction.

Eventually, for $b>a\geq 0$ and $i\in\{1,\ldots, n\}$, we will use the short-hand notation 
\begin{equation}\label{notation}
{\bf 1}^{i,n}_{[a,b]}:=(0,\ldots,0,{\bf 1}_{[a,b]},0,\ldots,0),
\end{equation}
where the indicator function ${\bf 1}_{[a,b]}$ is located in the $i$th position.

   \subsection{Rosenblatt-Wishart matrix}

In   this section,
we fix $n\geq 1$ and we let 
$\H_n$ denote the Hilbert space constructed in Section \ref{sec-fBm}, whose scalar product is defined by \eqref{scalar-fbm}.

\begin{prop}\label{prop31}
Fix $1\leq i,j\leq n$ and with the notation \eqref{notation}, consider
\begin{equation}\label{fij}
f^n_{ij}(d) =\frac{d}{2}\,\sum_{p=0}^{d-1} 
\left\{
{\bf 1}^{i,n}_{[\frac{p}{d},\frac{p+1}d]}\otimes {\bf 1}^{j,n}_{[\frac{p}{d},\frac{p+1}d]}
+
{\bf 1}^{j,n}_{[\frac{p}{d},\frac{p+1}d]}\otimes {\bf 1}^{i,n}_{[\frac{p}{d},\frac{p+1}d]}
\right\},\quad d\geq 1.
\end{equation}
Then $\{f^n_{ij}(d)\}_{d\geq 1}$ is a Cauchy sequence in $\H_n^{\otimes 2}$.
\end{prop}
\noindent{\it Proof}. We first observe that, as $d,d'\to\infty$,
\begin{equation}\label{riemann}
\frac{1}{dd'}\sum_{p,q=0}^{d,d'}\left(dd'\int_{p/d}^{(p+1)/d}du\int_{q/d}^{(q+1)/d}
dv|u-v|^{2H-2}\right)^2\to \int_{[0,1]^2}|u-v|^{4H-4}dudv.
\end{equation}
Now, let us compute $\langle f^n_{ij}(d),f^n_{ij}(d')\rangle_{\H^{\otimes 2}_n}$ for $d,d'\geq 1$:
\begin{eqnarray*}
\langle f^n_{ij}(d),f^n_{ij}(d')\rangle_{\H^{\otimes 2}_n}& =&\frac{dd'}{4}\sum_{p,q=0}^{d,d'}  \big\langle {\bf 1}^{i,n}_{[\frac{p}{d},\frac{p+1}d]}\otimes {\bf 1}^{j,n}_{[\frac{p}{d},\frac{p+1}d]},  {\bf 1}^{i,n}_{[\frac{q}{d'},\frac{q+1}{d'}]}\otimes {\bf 1}^{j,n}_{[\frac{q}{d'},\frac{q+1}{d'}]} \big\rangle_{\H^{\otimes 2}_n}\\
&+&\frac{dd'}{4}\sum_{p,q=0}^{d,d'}
\big\langle {\bf 1}^{i,n}_{[\frac{p}{d},\frac{p+1}d]}\otimes {\bf 1}^{j,n}_{[\frac{p}{d},\frac{p+1}d]}, {\bf 1}^{j,n}_{[\frac{q}{d'},\frac{q+1}{d'}]}\otimes {\bf 1}^{i,n}_{[\frac{q}{d'},\frac{q+1}{d'}]}  \big\rangle_{\H^{\otimes 2}_n}\\
&+&\frac{dd'}{4}\sum_{p,q=0}^{d,d'} \big\langle {\bf 1}^{j,n}_{[\frac{p}{d},\frac{p+1}d]}\otimes {\bf 1}^{i,n}_{[\frac{p}{d},\frac{p+1}d]},  {\bf 1}^{i,n}_{[\frac{q}{d'},\frac{q+1}{d'}]}\otimes {\bf 1}^{j,n}_{[\frac{q}{d'},\frac{q+1}{d'}]} \big\rangle_{\H^{\otimes 2}_n}\\
&+&\frac{dd'}{4}\sum_{p,q=0}^{d,d'} \big\langle {\bf 1}^{j,n}_{[\frac{p}{d},\frac{p+1}d]}\otimes {\bf 1}^{i,n}_{[\frac{p}{d},\frac{p+1}d]}, {\bf 1}^{j,n}_{[\frac{q}{d'},\frac{q+1}{d'}]}\otimes {\bf 1}^{i,n}_{[\frac{q}{d'},\frac{q+1}{d'}]} \big\rangle_{\H^{\otimes 2}_n}.
\end{eqnarray*}
We deduce
\begin{eqnarray*}
&&\langle f^n_{ij}(d),f^n_{ij}(d')\rangle_{\H^{\otimes 2}_n}\\ &=&\big(\frac12{\bf 1}_{\{i\neq j\}}+ {\bf 1}_{\{i=j\}}\big) H^2(2H-1)^2\frac{1}{dd'}\sum_{p,q=0}^{d,d'}\left(dd'\int_{p/d}^{(p+1)/d}du\int_{q/d}^{(q+1)/d}
dv|u-v|^{2H-2}\right)^2,
\end{eqnarray*}
implying in turn thanks to (\ref{riemann}) that 
\[
\langle f^n_{ij}(d),f^n_{ij}(d')\rangle_{\H^{\otimes 2}_n}\to \big(\frac12{\bf 1}_{\{i\neq j\}}+ {\bf 1}_{\{i=j\}}\big)H^2(2H-1)^2\int_{[0,1]^2}|u-v|^{4H-4}dudv\quad\mbox{as $d,d'\to\infty$}.
\]
The proof of Proposition \ref{prop31} is complete.\qed

\bigskip

We are now in a position to define the notion of Rosenblatt-Wishart matrix of size $n$ with Hurst parameter $H$.

\begin{definition}\label{defi-rose}
For each $1\leq i,j\leq n$, 
let $g^n_{ij}\in\H_n^{\otimes 2}$ be the limit of $\{f^n_{ij}(d)\}_{d\geq 1}$ given  in \eqref{fij}.
The $n\times n$ Rosenblatt-Wishart matrix with Hurst parameter\footnote{It is clear that the limiting kernels $g_{ij}^n$, $1\leq i\leq j\leq n$, depend on the Hurst parameter $H$.} $H$ is the  random symmetric matrix 
$\mathcal{R}^{(H)}_{n}=\big(  R_{ij} \big)_{1\leq i,j\leq n}$ with its entries given by $R_{ij}=I_2(g^n_{ij})$.
\end{definition}

  Equivalently, one can also define  $\mathcal{R}^{(H)}_{n}$ as the entrywise $L^2(\Omega)$-limit (as $d\to\infty$) of
\begin{equation}\label{snd}
 \mathcal{S}_{n,d}=\left( \begin{matrix}
    S_{11} (d)&  \ldots   & S_{1n}(d) \\
    \vdots &   &  \vdots  \\
    S_{n1}(d) &  \ldots   & S_{nn}(d)
\end{matrix}  \right),
\end{equation} where 
\begin{equation}\label{other}
S_{ij}(d)=
\left\{
\begin{array}{ll}
{\displaystyle d\sum_{p=0}^{d-1}(B^i_{\frac{p+1}{d}}-B^i_{\frac{p}d})(B^j_{\frac{p+1}{d}}-B^j_{\frac{p}d}) }&\quad\mbox{if $i\neq j$}\\
\quad\\
{\displaystyle d\sum_{p=0}^{d-1}\left\{(B^i_{\frac{p+1}{d}}-B^i_{\frac{p}d})^2-d^{-2H}\right\} }&\quad\mbox{if $i=j$}
\end{array}
\right..
\end{equation}
Indeed, bearing in mind the notation (\ref{fij}) and (\ref{other}), one obtains
\begin{equation}\label{limit@}
S_{ij}(d)=I_2(f^n_{ij}(d))\overset{L^2(\Omega)}{\to}
R_{ij}\quad \mbox{as $d\to\infty$, for any $1\leq i,j\leq n$ },
\end{equation}
where the existence of the previous $L^2(\Omega)$-limit is a consequence of 
Proposition \ref{prop31} and the isometry property of double Wiener-It\^o integrals.
It is clear from \eqref{limit@} and \eqref{other}    that  $\mathcal{R}^{(H)}_{n}$ satisfies the following \emph{compatibility} relation: if we delete the last row and last column of  $\mathcal{R}^{(H)}_{n+1}$, then we obtain a matrix that is distributed as  $\mathcal{R}^{(H)}_{n}$.

Another consequence of both (\ref{limit@}) and the explicit expression of $S_{ij}(d)$ 
is that the diagonal entries $R_{ii}$ of the Rosenblatt-Wishart matrix  are all independent from each other and  distributed according to the Rosenblatt distribution. We refer the reader to the survey \cite{VTaqqu13} and the references therein for the definition of the Rosenblatt distribution (one of them being precisely that it is the distributional limit of $S_{11}(d)$ as $d\to\infty$) together with its main properties (cumulants, characteristic function, etc.). For the non-diagonal entries, one first observes that
$(B^i,B^j)\overset{\rm law}{=}(\frac{B^i+B^j}{\sqrt{2}},\frac{B^i-B^j}{\sqrt{2}})$ as a process
if $i\neq j$,
so that
\[
\big\{S_{ij}(d)\big\}_{d\geq 1} \overset{\rm law}{=}
\left\{
\frac{d}2\sum_{p=0}^{d-1}
\left[(B^i_{\frac{p+1}{d}}-B^i_{\frac{p}d})^2-d^{-2H}\right]
-
\frac{d}2\sum_{p=0}^{d-1}
\left[(B^j_{\frac{p+1}{d}}-B^j_{\frac{p}d})^2-d^{-2H}\right]
\right\}_{d\geq 1},
\]
implying in turn, by letting $d$ go to infinity, that
\begin{equation}\label{identitylaw}
R_{ij} \overset{\rm law}{=}\frac{1}{2}\big[   R_{ii}+R_{jj}\big].
\end{equation}
Note however that the previous identity in law (\ref{identitylaw}) {\it only} holds
for {\it fixed} $i,j$, that is, the corresponding identity in law at the matrix level {\it does not} hold true.

\subsection{Proof of Theorem \ref{thm2}}

We let the notation of Theorem \ref{thm2} prevail, as well as the notation introduced in the previous section.    Before bounding the Wasserstein distance between $ \wh{\mathcal{W}}_{n,d}$ and $\mathcal{R}^{(H)}_{n}$ (with  $\wh{\mathcal{W}}_{n,d}$ given by (\ref{wnd2}) and $\mathcal{R}^{(H)}_{n}$  being the $n\times n$ Rosenblatt-Wishart matrix with Hurst parameter $H$), 
we observe the following two facts:
\begin{enumerate}
\item[(a)] Given (\ref{covariance}) and (\ref{frac}), one has 
$
\wh{\mathcal{W}}_{n,d} \overset{\rm law}{=}
d^{1-2H}\left(W_{ij}(d)
\right)_{1\leq i,j\leq n},
$
where 
\[
W_{ij}(d) = \left\{
\begin{array}{ll}
{\displaystyle \sum_{p=0}^{d-1}(B^i_{p+1}-B^i_{p})(B^j_{p+1}-B^j_{p}) }&\quad \mbox{if $i\neq j$}\\
\\
{\displaystyle \sum_{p=0}^{d-1}\left\{(B^i_{p+1}-B^i_{p})^2-1\right\} }&\quad\mbox{if $i=j$}
\end{array}
\right..
\]
By the \emph{selfsimilarity} property of fractional Brownian motion, we deduce
that 
$
\wh{\mathcal{W}}_{n,d} \overset{\rm law}{=}
\mathcal{S}_{n,d},
$
with $\mathcal{S}_{n,d}$ given by (\ref{snd}).
As a result, with the notation of Definition \ref{defi-rose}, we get
\[
 d_{\rm Wass}\big(  \wh{\mathcal{W}}_{n,d},  \mathcal{R}^{(H)}_{n} \big)
 = d_{\rm Wass}\big(  \mathcal{S}_{n,d},  \mathcal{R}^{(H)}_{n}  \big).
\]
\item[(b)] By its very definition, the Wasserstein distance is bounded by the $L^2(\Omega)$-distance, that is,
 \begin{equation}\label{bound2}
d_{\rm Wass}\big(  \mathcal{S}_{n,d},  \mathcal{R}^{(H)}_{n}   \big)
\leq \sqrt{\sum_{1\leq i, j\leq n}\E\Big[   \big(   S_{ij}(d)-R_{ij}  \big)^2\Big]}.
\end{equation}
\end{enumerate}

We are thus left to estimate the right-hand side of (\ref{bound2}).
For this, we refer to \cite{BretonNourdin08}:
in the inequality (17) therein, the existence of a finite constant $c_H>0$, depending only on $H$, satisfying
\[
\E\Big[   \big(   S_{ij}(d)-R_{ij}  \big)^2\Big]  \leq c_H d^{3-4H}
\]
is shown.
Plugging this into (\ref{bound2}) completes the proof of Theorem \ref{thm2}.\qed

\section{Further results}

\subsection{Relaxing the row independence: overall correlation}

In this section, we consider a more general setting where   we no longer assume the row independence. That is, the relation \eqref{covariance} will be replaced by a more general one, namely:
\begin{align}\label{COVrs}
\langle e_{ij}, e_{i'j'} \rangle_\H = r(i-i') s(j-j')
\end{align}
where $r$ is another correlation function also satisfying  $r(0)=1$.  Recall the definition \eqref{wnd} 
of $\mathcal{W}_{n,d}=(W_{ij})_{1\leq i,j\leq n}$. Since $\E[X_{ik}X_{jk}]=r(i-j)\neq {\bf 1}_{\{i=j\}}$ in general, 
its entries $W_{ij}$ are no more centered in general, so we should modify the corresponding Gaussian ensemble by shifting a little bit. Equivalently, by keeping the corresponding    Gaussian ensemble centered,  we can modify the Wishart ensemble accordingly, that is, we will consider the following \emph{shifted} Wishart matrix
 \begin{align}\label{shifted-W}
  \wt{\mathcal{W}}_{n,d} =  \big(  \wt{W}_{ij} \big) _{1\leq i, j\leq n} =  \big(  I_2(f_{ij}^{(d)})\big) _{1\leq i, j\leq n} 
   \end{align}
with kernels $f_{ij}^{(d)}$ given as in \eqref{Wndii}. That is, 
$$\wt{W}_{ij} =    \dfrac{1}{\sqrt{d}} \sum_{k=1}^d \big(X_{ik} X_{jk} - r(i-j) \big) \,.$$
Now let $\mathcal{G}_{n,d}^{(r,s)}=(G_{ij})_{1\leq i,j\leq n}$ be the $n\times n$ random symmetric matrix such that 
the associated random vector
$
(G_{11},\ldots,G_{1n},G_{21},\ldots,G_{2n},\ldots,G_{n1},\ldots,G_{nn})^T
$
of $\R^{n^2}$ is Gaussian and has the same covariance matrix as
$$
(\wt{W}_{11},\ldots,\wt{W}_{1n},\wt{W}_{21},\ldots,\wt{W}_{2n},\ldots,\wt{W}_{n1},\ldots,\wt{W}_{nn})^T.
$$
It is again routine to check that for $1\leq i\leq j\leq n$ and $1\leq u\leq v\leq n$,
\begin{align}\label{COV-Gnrsd}
\E\big[  G_{ij}G_{uv}  \big]= \frac{r(i-u)r(v-j)+ r(i-v)r(u-j) }{d} \sum_{k,\ell=1}^d  s(k-\ell)^2 \,,
\end{align}
and regardless of the integrability of $r$, the covariance in \eqref{COV-Gnrsd} is uniformed bounded by   $2 \| s \| ^2_{\ell^2(\Z)} $,
since $|r(k)|\leq 1$ for all $k$ and
\begin{equation}\label{tripledol}
\frac1d\sum_{k,\ell=1}^d s(k-\ell)^2 = \sum_{|j|\leq d} \left(1-\frac{|j|}{d}\right)s(j)^2.
\end{equation}
However,  it seems highly nontrivial to decide whether the covariance matrix of\footnote{See (\ref{half}) for the definition of the `half' of a symmetric matrix.}  $(\mathcal{G}_{n,d}^{(r,s)})^{\rm half}$  is invertible or not. Therefore, we will not be able to apply Proposition \ref{NPR10-thm} for the Gaussian approximation as we did in the proof of Theorem \ref{thm1}.  Instead, we shall use the following bounds from \cite[Theorem 6.1.2]{bluebook}  and  \cite[Theorem 9.3]{NZ17}, whose main interest for us is that the covariance
matrix of the underlying Gaussian vector may not be invertible. The price to pay, however, is that one can no longer deal with the Wasserstein distance, and we have to replace it by a smoother distance.

\begin{prop}\label{NPmbd} Fix  integers $m\geq 2$ and  $1\leq p_1\leq \ldots \leq p_m$.  Let $F = (F_1, \ldots, F_m)$ be a random vector such that $F_i = I_{p_i}(f_i)$, with some $f_i\in\H^{\odot p_i}$ for each $i$. Assume that $Z$ is a centered Gaussian vector in $\R^m$ with the {\it same} covariance matrix $C$ as $F$. Then,
\begin{itemize}
\item[(i)] for any $h:\R^m\to\R$ belonging to $C^2(\R^m)$ such that $\| h'' \| _\infty < +\infty$, we have 
\begin{align}
\bv\E [ h(F)] -\E[ h(Z)  ] \bv & \leq \frac{1}{2} \| h''\| _\infty \sum_{i,j=1}^m  \E\Big[ \bv C(i,j) - p_i^{-1} \langle DF_i, DF_j \rangle \bv \Big]  \label{00un} \\
&\leq  \frac{m}{2} \| h''\| _\infty \sqrt{ \sum_{i,j=1}^m  {\rm Var}\Big( \frac{1}{p_i} \langle DF_i, DF_j \rangle  \Big)  }  \,, \label{00sept}
\end{align}
where\footnote{The equation \eqref{00un} is clear from the proof of Theorem 6.1.2 in \cite{bluebook}, while the inequality \eqref{00sept} follows easily from the Cauchy-Schwarz: notice that  there is a typo in the display (6.1.3) of \cite{bluebook} and our version is correct.   }     $\| h'' \| _\infty : = \sup\Big\{ \,\, \bv \dfrac{\partial^2 h}{\partial x_i \partial x_j}(x) \bv\,:\, x\in\R^m, 1\leq i, j\leq m \Big\}$.
\item[(ii)]  for every $h\in C^2(\R^m)$ with $M_2(h): = \sup\big\{ \| D^2 h(x) \| _{\rm op}\,:\, x\in\R^m \big\} < +\infty$,
 $$  \Big\vert \E\big[ h(F) - h(Z) \big] \Big\vert  \leq  \frac{\sqrt{m} M_2(h) }{2p_1}\left(   \sum_{i,j=1}^m \Var\big( \langle DF_i, DF_j \rangle \big) \right)^{1/2}   \, .$$

\end{itemize}
\end{prop}

\begin{obs}  {\rm
 With the above result, it is natural to consider the following distances
\begin{align*}
d_2(X, Y): = \sup_{\| h'' \| _\infty \leq 1} \big|\E[h(X)] - \E[h(Y)] \big| \quad\text{and}\quad \wt{d_2}(X, Y): = \sup_{M_2(h) \leq 1} \big|\E[h(X)] - \E[h(Y)]\big| 
\end{align*}
for any $m$-dimensional random vectors $X,Y$ with square-integrable components.     Given $h\in C^2(\R^m)$,  it is easy to check that $\| h'' \| _\infty\leq \sqrt{m} M_2(h)$ and $M_2(h) \leq m \| h''\| _\infty$, so if $M_2(h)<+\infty$, then $h$ has at most quadratic growth so that the random variables $h(X)$ and $h(Y)$   are integrable.  It follows from the previous discussion that
\[
\frac{1}{\sqrt{m}} ~   \wt{d_2}(X,Y)  \leq d_2(X,Y) \leq m ~\wt{d_2}(X,Y) \,.
\]
}
\end{obs}

 Now we are ready to state the main result of this section.    

\begin{thm}\label{thm5} Let the above notation prevail, and assume that   $s\in \ell^2(\Z)$ and $r(0)= s(0)=1$.    Then,  {\rm (recalling the `half' notation from (\ref{half}),)}
\begin{itemize}
\item[(1)] we have
\begin{align*}
d_2\big( (\wt{\mathcal{W}}_{n,d})^{\rm half},  (\mathcal{G}_{n,d}^{(r,s)})^{\rm half} \big) =   \displaystyle O\Bigg\{\,\,   \frac{n^4}{\sqrt{d}} \Big(  \sum_{\vert k\vert \leq d} \vert s(k)\vert^{4/3}  \Big)^{3/2}  \,\,  \Bigg\} 
 \end{align*}
and 
\begin{align*}
\wt{d_2}\big((\wt{\mathcal{W}}_{n,d})^{\rm half},  (\mathcal{G}_{n,d}^{(r,s)})^{\rm half} \big)  =    \displaystyle O\Bigg\{\,\,   \frac{n^3}{\sqrt{d}} \Big(  \sum_{\vert k\vert \leq d} \vert s(k)\vert^{4/3}  \Big)^{3/2}  \,\,  \Bigg\} \,  ~;
 \end{align*}
\item[(2)] if in addition $r\in\ell^2(\Z)$, we have 
\begin{align*}
d_2\big( (\wt{\mathcal{W}}_{n,d})^{\rm half},  (\mathcal{G}_{n,d}^{(r,s)})^{\rm half} \big) =   O\Bigg\{ \frac{n^\frac72}{\sqrt{d}} \Big(  \sum_{\vert k\vert \leq d} \vert s(k)\vert^{4/3}  \Big)^{3/2}\Bigg\}  
 \end{align*}
and 
\begin{align*}
\wt{d_2}\big((\wt{\mathcal{W}}_{n,d})^{\rm half},  (\mathcal{G}_{n,d}^{(r,s)})^{\rm half} \big)  =  O\Bigg\{ \frac{n^\frac52}{\sqrt{d}} \Big(  \sum_{\vert k\vert \leq d} \vert s(k)\vert^{4/3}  \Big)^{3/2}\Bigg\}  \,  ~;
 \end{align*}

\item[(3)]  if in addition $r\in\ell^1(\Z)$, we have 
\begin{align*}
d_2\big( (\wt{\mathcal{W}}_{n,d})^{\rm half},  (\mathcal{G}_{n,d}^{(r,s)})^{\rm half}\big) =   O\Bigg\{ \frac{n^3}{\sqrt{d}} \Big(  \sum_{\vert k\vert \leq d} \vert s(k)\vert^{4/3}  \Big)^{3/2}\Bigg\}  
 \end{align*}
and 
\begin{align*}
\wt{d_2}\big((\wt{\mathcal{W}}_{n,d})^{\rm half},  (\mathcal{G}_{n,d}^{(r,s)})^{\rm half}\big)  =  O\Bigg\{ \frac{n^\frac52}{\sqrt{d}} \Big(  \sum_{\vert k\vert \leq d} \vert s(k)\vert^{4/3}  \Big)^{3/2}\Bigg\}  \,  ~.
 \end{align*}

\end{itemize}

\end{thm}

\begin{proof} For $i,j,p,q\in\{1, \ldots, n\}$, we have, similarly to (\ref{doubledol}), that
\begin{align}
&\quad \big\| f_{ij}^{(d)}\otimes_1 f_{pq}^{(d)} \big\| ^2_{\H^{\otimes 2}} \notag \\
&=\frac{1}{16d^2} \left\| \sum_{k,\ell=1}^d \big( e_{ik}\otimes e_{jk} + e_{jk}\otimes e_{ik}\big)\otimes_1 \big(e_{p\ell}\otimes e_{q\ell} + e_{q\ell}\otimes e_{p\ell}\big) \right\| ^2_{\H^{\otimes 2}} \notag \\
& = \frac{1}{16d^2} \Bigg\| \sum_{k,\ell=1}^d  \bigg( (e_{ik}\otimes e_{p\ell}) r(j-q) s(k-\ell) + (e_{ik}\otimes e_{q\ell}) r(p-j) s(k-\ell)  \notag \\
 & \qquad\qquad\qquad  +  (e_{jk}\otimes e_{p\ell}) r(i-q) s(k-\ell) + e_{jk}\otimes e_{q\ell} r(i-p) s(k-\ell)  \bigg)  \Bigg\| ^2_{\H^{\otimes 2}} \notag \\
 & =  \frac{ \X_{ i,j,p,q}}{16d^2} \sum_{k,\ell, u, v=1}^d s(k-u)s(\ell-v) s(k-\ell)s(u-v)  \leq \frac{\X_{ i,j,p,q}}{16d} \left(  \sum_{\vert k\vert \leq d} \vert s(k)\vert^{\frac43}  \right)^3  \,, \label{crucial-0}
\end{align}
where the last inequality follows from {\bf\small Fact 3} in Section 2.1 and  $ \X_{ i,j,p,q}$ is a sum of the following sixteen terms:
\begin{align*}
  &\quad \quad  r(j-q)^2 + r(j-q)r(p-j)r(p-q) + r(j-i)r(q-j)r(i-q)\\
  &\quad   + r(i-j) r(p-q)  r(i-p) r( j- q) + r(p-j)^2 + r(p-q)r(j-q) r(p-j) \\
 &\quad    +  r(i-j) r(p-q) r(i-q) r(p-j) + r(i-j) r(i-p)r(p-j) + r(i-q)^2 + r(i-p)^2\\
 &\quad    + r(i-q) r(i-j) r(j-q) + r(i-q) r(i-j) r(p-q) r(p-j) + r(i-q) r(p-q) r(i-p)  \\
  &\quad   + r(i-p) r(i-j) r(p-q)r(j-q)  +r(i-p)r(i-j) r(p-j)+r(i-p)r(p-q)r(q-i) \, .
\end{align*}
Using the fact that $\vert r(k)\vert \leq 1$ for each $k\in\Z$ and $2ab \leq a^2 + b^2$ for $a,b\in\R$, we have the following estimate:
\begin{align}
 \X_{ i,j,p,q} = O\Big(  r(j-q)^2 + r(i-j)^2 + r(i-p)^2 + r(i-q)^2 + r(j-p)^2 + r(p-q)^2 \Big) \,, \label{X-bdd}
\end{align}
while the following two  crude estimates also hold:
 \begin{align}
  \X_{ i,j,p,q} &= O\big[ \vert r(j-q)\vert+\vert r(i-j)\vert +\vert r(i-p)\vert +\vert r(i-q)\vert +\vert r(j-p)\vert +\vert r(p-q)\vert \big] \,, \label{crude-1}\\
\bv \X_{ i,j,p,q} \bv& \leq 16\,. \label{crude-2}
 \end{align}
 Note that if $r\in\ell^2(\Z)$,
 \begin{align}\label{cc1}
 \sum_{i , j , p , q =1}^n r(j-q)^2 = n^2 \sum_{1\leq j  , q\leq n} r(j-q)^2 = n^3 \dfrac{1}{n} \sum_{k,\ell=1}^n r(k-\ell)^2  \leq n^3 \| r\| _{\ell^2(\Z)}^2 \,, 
  \end{align}
where the last inequality follows from (\ref{tripledol}).  The same argument will give us that under the assumption $r\in\ell^1(\Z)$, 
\begin{align}\label{cc2}
 \sum_{i , j , p , q =1}^n  \vert r(j-q)\vert   \leq n^3 \| r\| _{\ell^1(\Z)}  \,.
        \end{align}

Now let us show the bounds in the assertion (1):  it follows first from  Proposition \ref{NPmbd}-(i), then from {\bf\small Fact 1} and isometry relation  that 
\begin{align}
d_2\big( (\wt{\mathcal{W}}_{n,d})^{\rm half},  (\mathcal{G}_{n,d}^{(r,s)})^{\rm half} \big)& \leq\frac{1}{2}    \sum_{\substack{ 1\leq i \leq j \leq n \\  1\leq p \leq q \leq n}} \E\big[ \vert 2 I_2(f_{ij}^{(d)} \otimes_1 f_{pq}^{(d)} ) \vert \big]  \notag \\
&  \leq  \sqrt{2}   \sum_{\substack{ 1\leq i \leq j \leq n \\  1\leq p \leq q \leq n}}  \big\| f_{ij}^{(d)} \otimes_1 f_{pq}^{(d)} \big\| _{\H^{\otimes 2}} . \label{there!}
\end{align}
Then the crude estimate \eqref{crude-2} and the bound \eqref{crucial-0} imply  that 
\[
d_2\big( (\wt{\mathcal{W}}_{n,d})^{\rm half},  (\mathcal{G}_{n,d}^{(r,s)})^{\rm half} \big) \leq  \frac{n^4}{\sqrt{d}}   \left(  \sum_{\vert k\vert \leq d} \vert s(k)\vert^{4/3}  \right)^{3/2}.
\]
And similarly,  we can apply Proposition \ref{NPmbd}-(ii)  to get
\begin{align}
\wt{d_2}\big( (\wt{\mathcal{W}}_{n,d})^{\rm half},  (\mathcal{G}_{n,d}^{(r,s)})^{\rm half}  \big)& \leq \frac{n}{4} \sqrt{    \sum_{\substack{ 1\leq i \leq j \leq n \\  1\leq p \leq q \leq n}}     \Var\big( 4 I_2(f_{ij}^{(d)} \otimes_1 f_{pq}^{(d)} )  \big)  }  \notag\\
&  \leq \sqrt{2} n  \sqrt{    \sum_{\substack{ 1\leq i \leq j \leq n \\  1\leq p \leq q \leq n}}   \| f_{ij}^{(d)} \otimes_1 f_{pq}^{(d)} \| ^2_{\H^{\otimes 2}} }   \label{here!}  \\
&\leq     \sqrt{2} n  \sqrt{  \frac{n^4}{d}   \left(  \sum_{\vert k\vert \leq d} \vert s(k)\vert^{4/3}  \right)^{3}   }    .\notag
\end{align}

Now let us assume  additionally that $r\in\ell^2(\Z)$, then with similar arguments and using in particular the estimates \eqref{X-bdd}, \eqref{cc1}, we have
\begin{align*}
&\qquad d_2\big( (\wt{\mathcal{W}}_{n,d})^{\rm half},  (\mathcal{G}_{n,d}^{(r,s)})^{\rm half}  \big) \\
&\leq  \frac{n^2}{2} \sqrt{  \sum_{\substack{ 1\leq i \leq j \leq n \\  1\leq p \leq q \leq n}}   \| f_{ij}^{(d)} \otimes_1 f_{pq}^{(d)} \| ^2_{\H^{\otimes 2}}      } =  O(n^2)     \sqrt{  \sum_{\substack{ 1\leq i \leq j \leq n \\  1\leq p \leq q \leq n}} \frac{ \X_{ i,j,p,q}}{16d} \left(  \sum_{\vert k\vert \leq d} \vert s(k)\vert^{4/3}  \right)^3  } \\
&=O(n^2)     \sqrt{ \frac{n^3 \|  r \| ^2_{\ell^2(\Z)}}{d} \left(  \sum_{\vert k\vert \leq d} \vert s(k)\vert^{4/3}  \right)^3  } \quad\text{by  \eqref{cc1};   }
\end{align*}
and 
\begin{align*}
\wt{d_2}\big( (\wt{\mathcal{W}}_{n,d})^{\rm half},  (\mathcal{G}_{n,d}^{(r,s)})^{\rm half} \big) &   \leq \sqrt{2} n  \sqrt{    \sum_{\substack{ 1\leq i \leq j \leq n \\  1\leq p \leq q \leq n}}   \| f_{ij}^{(d)} \otimes_1 f_{pq}^{(d)} \| ^2_{\H^{\otimes 2}} }   \quad\text{see \eqref{here!};}   \\
& =    O\left\{ \frac{n^\frac52}{\sqrt{d}} \left(  \sum_{\vert k\vert \leq d} \vert s(k)\vert^{4/3}  \right)^\frac32\right\}   \,.
\end{align*}
Now let us assume  additionally that $r\in\ell^1(\Z)$, then with similar arguments and using in particular the estimates \eqref{crude-2}, \eqref{cc2}, we have
\begin{align*}
d_2\big( (\wt{\mathcal{W}}_{n,d})^{\rm half},  (\mathcal{G}_{n,d}^{(r,s)})^{\rm half}  \big) &\leq   \sqrt{2}   \sum_{\substack{ 1\leq i \leq j \leq n \\  1\leq p \leq q \leq n}}  \| f_{ij}^{(d)} \otimes_1 f_{pq}^{(d)} \|_{\H^{\otimes 2}}  \quad\text{see \eqref{there!};} \\
& \leq      \sum_{\substack{ 1\leq i \leq j \leq n \\  1\leq p \leq q \leq n}}   \frac{\sqrt{ \X_{ i,j,p,q} }}{\sqrt{d}} \left(  \sum_{\vert k\vert \leq d} \vert s(k)\vert^{4/3}  \right)^{3/2}   \quad\text{by \eqref{crucial-0};} \\
& = O \left(  \sum_{\vert k\vert \leq d} \vert s(k)\vert^{4/3}  \right)^{3/2} \times \frac{1}{\sqrt{d}}   \sum_{i,j,p,q=1}^n   \vert r(j-q)\vert  \quad\text{by \eqref{X-bdd};}\\
& = O \left(  \sum_{\vert k\vert \leq d} \vert s(k)\vert^{4/3}  \right)^{3/2} \times \frac{n^3}{\sqrt{d}}  \| r\| _{\ell^1(\Z)} \quad\text{by \eqref{cc2}.} 
\end{align*}
   Finally,
\begin{align*}
\wt{d_2}\big((\wt{\mathcal{W}}_{n,d})^{\rm half},  (\mathcal{G}_{n,d}^{(r,s)})^{\rm half}  \big) &   \leq \sqrt{2} n  \sqrt{    \sum_{\substack{ 1\leq i \leq j \leq n \\  1\leq p \leq q \leq n}}   \| f_{ij}^{(d)} \otimes_1 f_{pq}^{(d)} \| ^2_{\H^{\otimes 2}} }   \quad\text{see \eqref{here!};}   \\
& = \sqrt{2} n  \sqrt{    \sum_{\substack{ 1\leq i \leq j \leq n \\  1\leq p \leq q \leq n}}  \frac{\X(i,j,p,q)}{d}  \left(  \sum_{\vert k\vert \leq d} \vert s(k)\vert^{4/3}  \right)^{3}  }   \quad\text{by \eqref{crucial-0};} \\
& =  O(n)  \left(  \sum_{\vert k\vert \leq d} \vert s(k)\vert^{4/3}  \right)^{3/2}\times  \sqrt{ \frac{1}{d} \sum_{i,j,p,q=1}^n  \vert r(i-j)\vert} \quad\text{by \eqref{crude-1};} \\
& = O(n)  \left(  \sum_{\vert k\vert \leq d} \vert s(k)\vert^{4/3}  \right)^{3/2}\times \sqrt{ \frac{n^3}{d}\| r\| _{\ell^1(\Z)}} \qquad\text{by \eqref{cc2}.}\qedhere 
\end{align*}

\end{proof}

Keeping in mind that our goal is to obtain high-dimensional regime for the distributional convergence, we compare the distances $d_2$, $\wt{d_2}$ with the Wasserstein distance  in the following remark.

\begin{prop}\label{rk44}
Let $X,Y$ be two random vectors in $\R^m$ with square-integrable components. Then, by standard \emph{smoothing argument}, we have 
\begin{align}\label{smoothie}
d_{\rm Wass}(X, Y) \leq 2\sqrt{2}  \,m^{1/4}\,\sqrt{d_2(X, Y)   } \,.
\end{align}
\end{prop}
\begin{proof}
Indeed, consider any  $h:\R^m\to\R$ $1$-Lipschitz function. Define, for $\e > 0$,  the smoothed version $h_\e(x) := \E\big[ h(x + \e N ) \big]$ for every $x\in\R^m$, with $N$ standard Gaussian independent of $X$ and $Y$. It is routine to check via a simple Gaussian integration by parts that for each $i,j\in\{1,\ldots, m\}$
\[
\frac{\partial^2}{\partial x_i \partial x_j} h_\e(x) = \frac{1}{\e} \E\Big[ N_i \frac{\partial}{\partial x_j} h(x + \e N ) \Big] \,\,,
\]
from which we obtain  $\| h_\e'' \| \leq \e^{-1}$. Therefore, we have $\big\vert  \E[ h_\e(X) - h_\e(Y)  ]  \big\vert  \leq   \e^{-1}d_2(X, Y)$ while 
$ \big\vert  \E[ h(X) - h_\e(X)  ]  \big\vert  \leq   \E\big[  \vert h(X) - h(X + \e N)  \vert  \big]   \leq \e \,\E\big[ \| N \| _{\R^m} \big] \leq \e\, \sqrt{m}$, thus we have
 \begin{align}
\big\vert \E[ h(X) - h(Y) ]  \big\vert &\leq  \big\vert\E[ h(X) - h_\e(X) ]\big\vert +\big\vert\E[ h_\e(X) - h_\e(Y) ] \big\vert+ \big\vert\E[ h_\e(Y) - h(Y)]\big\vert  \label{LHS0}\\
& \leq 2 \e \,\sqrt{m} +  \e^{-1}\, d_2(X, Y)  \,\,\,\text{for any $\e > 0$}\,. \label{optim0}
\end{align}
 Optimizing over $\e > 0$ in \eqref{optim0} first, then taking supremum on the left-hand side of \eqref{LHS0} give us \eqref{smoothie}.
\end{proof}


As an easy consequence of Proposition \ref{rk44} combined with Lemma \ref{22} and Theorem \ref{thm5}, we have the following bounds in Wasserstein distance.

\begin{cor} \label{cor55}   Let the assumptions of Theorem \ref{thm5} prevail, that is,   $s\in \ell^2(\Z)$ and $r(0)= s(0)=1$.    Then, 
\begin{align*}
d_{\rm Wass}    \big( \wt{\mathcal{W}}_{n,d}  ,\mathcal{G}_{n,d}^{(r,s)}\big) = \begin{cases}
{\displaystyle  O\Bigg\{ \frac{n^{5/2}}{d^{1/4}} \Big(  \sum_{\vert k\vert \leq d} \vert s(k)\vert^{4/3}  \Big)^{3/4}\Bigg\}     } \, ~ ;\\
 {\displaystyle  O\Bigg\{ \frac{n^{9/4}}{d^{1/4}} \Big(  \sum_{\vert k\vert \leq d} \vert s(k)\vert^{4/3}  \Big)^{3/4}\Bigg\}   } \,\,\,~ \text{if in addition $r\in\ell^2(\Z)$;}\\
 {\displaystyle  O\Bigg\{ \frac{n^{2}}{d^{1/4}} \Big(  \sum_{\vert k\vert \leq d} \vert s(k)\vert^{4/3}  \Big)^{3/4}\Bigg\}   } \,\,\,~ \text{if in addition $r\in\ell^1(\Z)$.}\\
\end{cases}
 \end{align*}
In particular, if in addition $s\in \ell^{4/3}(\Z)$, then  $ \wt{\mathcal{W}}_{n,d}$ is $\phi$-close to $\mathcal{G}_{n,d}^{(r,s)}$ for 
\begin{align*}
\phi(n,d) = \begin{cases}
n^{10}/d \, ~ ;\\
n^{9}/d\,\,\,\quad \text{if in addition $r\in\ell^2(\Z)$;}\\
n^{8}/d\,\,\,\quad \text{if in addition $r\in\ell^1(\Z)$.}\\
\end{cases}
 \end{align*}

\end{cor}

\medskip

\subsection{Random $p$-tensors}   In this section, we  investigate the Gaussian approximation of  random $p$-tensors, a natural extension of the Wishart matrices.  The distribution of the random $p$-tensors has recently gained        interest in the context of machine learning, see for instance \cite{AGHKT14}. 

\medskip

Let us suppose $\{\e_i, i=1,\ldots, n\}$ is the canonical basis of $\R^n$. Then for $p\geq 2$, the $p$-tensor product of $x\in\R^n$ is given by  
  $$ x^{\otimes p} =  \left( \,\,  \sum_{i=1}^n \langle x , \e_i \rangle_{\R^n} \e_i \,\,\right)^{\otimes p}  =   \sum_{i_1,\ldots, i_p=1}^n \left(  \prod_{k\in[p]}\langle x , \e_{i_k} \rangle_{\R^n}   \right) \e_{i_1}\otimes \cdots \otimes\e_{i_p} ,  $$
 which can be identified as a vector in $\R^{n^p}$.  For what we care most in the present work, $\mathbb{X}_i = ( X_{1i}, \ldots, X_{ni} )^T \in\R^n$ is the $i$th column of  the rectangular random matrix $\mathcal{X}_{n,d}$, then
  $$\mathbb{X}_i^{\otimes p} =  \sum_{j_1,\ldots, j_p=1}^n \left(  \prod_{k=1}^p  X_{j_k i }   \right) \e_{j_1}\otimes \cdots \otimes\e_{j_p}    $$
 so that 
$$  \dfrac{1}{\sqrt{d}} \sum_{i=1}^d    \mathbb{X}_i^{\otimes p}   =   \sum_{j_1,\ldots, j_p=1}^n  \dfrac{1}{\sqrt{d}} \sum_{i=1}^d  \left(  \prod_{k=1}^p  X_{j_k i }   \right) \e_{j_1}\otimes \cdots \otimes\e_{j_p}  \,,$$
which can also be seen as the random vector $ \mathbf{Y}   \in \R^{n^p} $ given below  $$\left( \,\,  \mathbf{Y}_{\textbf{j}} =\dfrac{1}{\sqrt{d}} \sum_{i=1}^d     \prod_{k=1}^p  X_{j_k i } , 1\leq j_1,\ldots, j_p \leq n    \,\, \right) \,\, .$$
Note that $2$-tensor case corresponds to the matrix, which we have dealt with in the previous sections. To simplify the matter and without losing the essence, we will remove the diagonal terms, that is, we will only consider the Gaussian approximation of 
\begin{align}\label{tensor0diag}
 \mathcal{Y}_{n,d}= \left( \,\,  \mathbf{Y}_{\textbf{j}} :=\dfrac{1}{\sqrt{d}} \sum_{i=1}^d     \prod_{k=1}^p  X_{j_k i } \,,\, \textbf{j}\in\Delta_p   \,\, \right)  \,\, ,
\end{align}
where $\Delta_p : = \big\{ \{ j_1,\ldots, j_p\}\in \{1, \ldots, n\}^p \,:\,  j_1, \ldots, j_p$ are mutually distinct$\big\}$.

Although our approach can extend to a more general setting, we focus on the full independence case to illustrate the ideas. That is, from now on, we assume that  
\begin{center} {\it  the entries of $\mathcal{X}_{n,d}$ are i.i.d. standard Gaussian. }
\end{center}
And in the following result, we present the high-dimensional regime (as $n,d$ both tend to infinity), in which the random $p$-tensor \eqref{tensor0diag}  is close to a Gaussian distribution.

\begin{thm} Fix an integer $p\geq 2$ and let the above assumptions and notation prevail. Let $Z_\nu$ be a standard Gaussian vector in $\R^\nu$ with $\nu  = p! {n\choose p}$. Then we have 
\begin{align}\label{tensorbdd}
  d_{\rm Wass}\big(   \mathcal{Y}_{n,d}, Z_\nu \big) = O\big(  \sqrt{n^{2p-1}/d   } \,\big) \,. 
\end{align}
In particular, $\mathcal{Y}_{n,d}$ is close to a standard Gaussian vector with the same dimension as soon as $n^{2p-1}/d \to 0$.

\end{thm}

\noindent{\it Proof.}  We split the proof into two main steps. The first step is to restrict our attention to the random vector in $\R^{\nu/p!}$
\begin{align}\label{uparrow}
\mathcal{Y}_{n,d}^{\uparrow}= \big( \mathbf{Y}_{\textbf{j} }, \textbf{j}\in\Delta^\uparrow_p  \big)
\end{align}
and the standard Gaussian vector $Z_\nu^\uparrow$ in   $\R^{\nu/p!}$, here $\Delta^\uparrow_p= \big\{ \textbf{j}\in \Delta_p\,:\, j_1 < j_2 < \ldots < j_p \big\}$.

\medskip

\underline{\it Step 1}:  \emph{Gaussian approximation of $\mathcal{Y}_{n,d}^{\uparrow}$}.    This step can be established easily using the multivariate bound in Proposition \ref{NPR10-thm}.  To do so, we first write for $\textbf{j}\in\Delta^{\uparrow}_p$,  
$$Y_{\textbf{j}} = {\displaystyle \dfrac{1}{\sqrt{d}} \sum_{i=1}^d     \prod_{k=1}^p  X_{j_k i } }  =  I_p(f_{\textbf{j}}^{(d)})  $$
 with the kernel
 $
 f_{\textbf{j}}^{(d)}  =  \frac{1}{\sqrt{d}} \sum_{i=1}^d   \text{sym} \big( e_{j_1 i }\otimes \cdots \otimes   e_{j_p i } \big)
 $,
where $\text{sym}(\cdot)$ stands for the canonical symmetrization.   Then for any $r\in\{ 1, \ldots, p-1\}$ and $\textbf{j}, \textbf{j}' \in \Delta^\uparrow_p$,  it is immediate to verify by using the orthogonality of $\{ e_{ij}, i,j\geq 1\}$  that
$$
\big\| f_{\textbf{j}}^{(d)}  \otimes_r f_{\textbf{j}'}^{(d)}  \big\| _{\H^{\otimes 2p-2r}}^2  = O(1/d)
$$
and $\big\| f_{\textbf{j}}^{(d)}  \otimes_r f_{\textbf{j}'}^{(d)}  \big\| _{\H^{\otimes 2p-2r}}^2 = 0$ if additionally $\textbf{j}$ and $\textbf{j}'$ have no common index (\emph{i.e.} $\textbf{j}\cap\textbf{j}' =\emptyset$).  

Therefore, in view of the equation (6.2.3) in \cite{bluebook}, we have 
\[
\Var\Big(  \big\langle DI_p(f_{\textbf{j}}^{(d)}),D I_p(f_{\textbf{j}'}^{(d)}) \big\rangle_\H \Big) = O\left( \sum_{r=1}^{p-1} \big\| f_{\textbf{j}}^{(d)}  \otimes_r f_{\textbf{j}'}^{(d)}  \big\| _{\H^{\otimes 2p-2r}}^2  \right) = O(1/d)
\]
and $\Var\big(  \langle DI_p(f_{\textbf{j}}^{(d)}), D I_p(f_{\textbf{j}'}^{(d)})  \rangle_\H \big) =0$  if additionally  $\textbf{j}\cap\textbf{j}' =\emptyset$.   Thus, it follows from Proposition  \ref{NPR10-thm} that 
\begin{align*}
d_{\rm Wass}\big(  \mathcal{Y}_{n,d}^\uparrow,  Z_\nu^\uparrow \big) \leq \left(  \sum_{ \textbf{j},  \textbf{j}' \in\Delta^\uparrow_p  } \Var\Big(  p^{-1}  \big\langle DI_p(f_{\textbf{j}}^{(d)}),D I_p(f_{\textbf{j}'}^{(d)}) \big\rangle_\H \Big)    \right)^{1/2} \end{align*}
and the sum $ \sum_{ \textbf{j},  \textbf{j}' \in\Delta^\uparrow_p  }$ can be replaced by the sum  $\sum_{ \textbf{j},  \textbf{j}' \in\Delta^\uparrow_p :  \textbf{j}\cap\textbf{j}' \neq\emptyset }$. It is easy to see that the cardinality of the set 
$
\big\{ (j_1, \ldots, j_p, j_1', \ldots, j_p')\in\{1,\ldots, n\}^{2p} :\, \textbf{j}, \textbf{j}'\in\Delta_p^\uparrow ~\text{and}~ \textbf{j}\cap\textbf{j}' \neq\emptyset \big\}
$
is $O(n^{2p-1})$.  Hence we conclude our \underline{\it Step 1} with $d_{\rm Wass}\big(  \mathcal{Y}_{n,d}^\uparrow,  Z_\nu^\uparrow \big) =  O\big(  \sqrt{n^{2p-1}/d   } \,\big)$.

\medskip

\underline{\it Step 2}:  \emph{The passage from  $\mathcal{Y}_{n,d}^\uparrow$ \eqref{uparrow} to $\mathcal{Y}_{n,d}$ \eqref{tensor0diag}}.   This can be done by using an easy extension of     Lemma \ref{22}.  For the sake of completeness, we sketch the main arguments  below (Recall $\nu = p!{n\choose p}$.) 
\begin{itemize}
\item Consider  $\mathbf{x}= (\mathbf{x}_{\textbf{j}}, \textbf{j}\in\Delta_p )$, let   $h:\R^\nu\to\R$ be $1$-Lipschitz;

 \item  set  $h^\uparrow(\mathbf{x}^\uparrow) = \frac{1}{\sqrt{p!}} h(\mathbf{x})$ for any $\mathbf{x}= (\mathbf{x}_{\textbf{j}}, \textbf{j}\in\Delta_p)$ and it is easy to verify that $h^{\uparrow}$ is a $1$-Lipschitz function defined on $\R^{\nu/p!}$.

 \end{itemize}
Then $ \vert \E [ h(\mathcal{Y}_{n,d}) - h(Z_\nu)  ]  \vert  = \sqrt{p!} \,  \vert  \E [  h^\uparrow(\mathcal{Y}^\uparrow_{n,d})  - h^\uparrow(Z_\nu^\uparrow)  ]  \vert \leq  \sqrt{p!} \, d_{\rm Wass} (  \mathcal{Y}_{n,d}^\uparrow,  Z_\nu^\uparrow  ) = O (  \sqrt{n^{2p-1}/d   } )$, hence the desired bound \eqref{tensorbdd} follows immediately. \qed
 
 \medskip

Note that the above theorem is a substantial extension of the  regime in \cite{BDER16,JiangLi13}.

\section{Conclusion and open problems}

  In our work, we answered some questions meanwhile we also raised some questions. To motivate further research, we  provide a summary in this section. 
   
   \medskip

 In a large complex system formed by many independent components, many universal pictures can appear, for example the classical central limit theorem or Gaussian fluctuation  is one of them.  In this work, the large complex object we are considering is the Wishart matrix $\mathcal{W}_{n,d}$ \eqref{wnd}, which appears naturally as the sample covariance matrix in the context of multivariate analysis. It is also closely related to the so-called ``{\it principal component analysis}'', see \cite{Johnstone06}. Several previous papers \cite{BDER16, BG16, JiangLi13} have been devoted to the high dimensional limit and their methodology  consists of random matrix techniques or information-theoretical tools. However, these articles only considered the case of full independence, that is, the entries of the rectangular random matrix $\mathcal{X}_{n,d}$ that forms $\mathcal{W}_{n,d}$ are independent. Such a setting gives arise to several advantages, for example, 
 \begin{itemize}
 \item
 in the Gaussian setting, the authors of \cite{BDER16} were able to directly compute the total variation distance between $\mathcal{W}_{n,d}$ and corresponding Gaussian ensemble using the available density formulae; 
 \item
 the authors of \cite{BG16} were able to perform an induction argument in order to use the entropic CLT as well as some other tools to bound the total variation distance. 
 \end{itemize}
 The lack of independence breaks the above strategies, while it is well known that Stein's method of distributional approximation is very powerful for investigating situations in presence of dependence.  This motivated us to apply Stein's method for  studying the high-dimensional limit of large Gaussian correlated  Wishart matrices. In the present work,   we not only recover known high-dimensional regimes but also provide new phenomena in the correlated case, see our Theorem \ref{thm1} and many interesting examples in Corollary \ref{cor-fbm}.    We are also able  to   deal with non-central limit in high-dimensional regime, see Theorem \ref{thm2}, where   notably a new probabilistic object called the Rosenblatt-Wishart matrix shows up as  the highlight in our Section 3.  Our bounds are described in terms of the Wasserstein distance, different from the total variation distance in \cite{BDER16,BG16}.   As we simply consider the symmetric matrix of size $n$ as the $n^2$-dimensional vector,   the random vector $\mathcal{W}_{n,d}$ has many repeated components so that it has no density on $\R^{n^2}$, preventing us from obtaining the bounds in the total variation distance;  on the other hand, (even for the Gaussian approximation of $\mathcal{W}_{n,d}^{\rm half}$) it is extremely difficult to get a \emph{multivariate} total variation bound by using Stein's method. 
 
 \medskip
 
 In Section 4, we obtain high-dimensional regimes in the case of overall correlation. It covers the case of column independence  where the column vectors $\mathbb{X}_i   = ( X_{1i}, \ldots, X_{ni} )^T$, $i=1,\ldots, d$, of the rectangular random matrix $\mathcal{X}_{n,d}$ are independent and identically distributed. As pointed out in the paper \cite{BG16},  a straightforward application of the Stein's method (see \emph{e.g.}\cite{CM08})    gives us the following bound in the \emph{full independence case}: 
 \begin{align}
 \wt{d_2}\big(  \mathcal{W}_{n,d}, \mathcal{Z}_n \big) = O\big(n^3/\sqrt{d}\big) \,,  \label{oldreg}   
 \end{align}
 where $\mathcal{W}_{n,d}$ is given by \eqref{wnd} (with general entry distribution) and $\mathcal{Z}_n$ belongs to the GOE described in \eqref{gn}. The authors of \cite{BG16} then stated an intriguing open question: whether there is a way to use Stein's method to recover their regime ($d$ much larger than $n^3$) in any reasonable metric (total variation metric, Wasserstein metric, etc.)?  So in the present paper, we answered this open question in the \emph{Gaussian setting} and moreover, what we achieved is beyond the  full independence case: as already mentioned in \cite{BG16},  the induction argument therein (at the core of their strategy) breaks   without the full independence assumption.

   It is also worth mentioning that the authors of \cite{BG16} claimed that the case of row independence is probably much harder than column independence. This claim is reasonable in view of how they applied   Stein's method to get \eqref{oldreg}:  they first express the Wishart matrix as the normalized partial sum of i.i.d $2$-tensors
$$\mathcal{W}_{n,d} = \frac{1}{\sqrt{d}} \sum_{i=1}^d  \mathbb{Y}_i :=  \frac{1}{\sqrt{d}} \sum_{i=1}^d   \big(  \mathbb{X}_i^{\otimes 2}  - \text{diag} \big(\mathbb{X}_i^{\otimes 2} \big)   \big) $$
and then they directly applied the bound in \cite{CM08} to obtain \eqref{oldreg}.  For us, it is much easier to deal with the case of row independence and more importantly, we get better regimes. Such a difference is not a contradiction to the claim from \cite{BG16}, but instead it stems  from our different strategy of applying Stein's method:  we first consider the half-matrix or the random vector formed by the upper-triangular entries, which, in the case of row-independence (not in the case of column independence), has invertible covariance matrix, so we are able to use a powerful machinery --- the so-called Malliavin-Stein approach, to get the high-dimensional regimes for half-matrix. And the high-dimensional regime for the full-size matrix can be easily passed from that of half-matrix in view of our easy Lemma \ref{22}. This trick has also been applied in Section 4.2 to obtain the high-dimensional regime of random $p$-tensors, a natural extension of Wishart matrices.\\

To conclude this article, we propose several open questions:

\medskip

  \underline{\bf Q.1}: We introduced the  notion of Rosenblatt-Wishart matrix   in our Section 3.2 and we listed several basic properties of such a new object.  
  \begin{itemize}
 \item[]  Is   it the right candidate for the {\it random-matrix approximation} (in the sense of \cite{PPPA16}) of the so-called \emph{non-commutative Tchebycheff process} introduced in \cite{NourdinTaqqu14} ?   
  \end{itemize}

      \vspace{1mm}

  \underline{\bf Q.2}: In Corollary  \ref{cor55}, we provide the high-dimensional regimes with respect to the Gaussian approximation in the case of overall correlation.    Similar to Corollary \ref{cor-fbm}, one may be able to construct many interesting examples of the correlation functions $r$ and $s$ that arise from some generalization of fractional Brownian sheets. 
  This is left for interested readers.

      \vspace{1mm}

  \underline{\bf Q.3}: Following {\bf Q.2}, one may ask the following question:
    \begin{itemize}
 \item[]   what does the non-central high-dimensional limit look like in the case of overall correlation? Can one obtain some generalization of our Rosenblatt-Wishart matrix ?
  \end{itemize}

      \vspace{1mm}

  \underline{\bf Q.4}: In the case of full independence, we obtain the high-dimensional regime for the random $p$-tensors with respect to the Gaussian approximation. So  one may ask the following   reasonable question:
     \begin{itemize}
 \item[]   Can one formulate a natural correlated setting for random $p$-tensors?   Can  one still obtain nice high-dimensional regimes therein?  Concerning the non-central high-dimensional limit, what is the generalization of Rosenblatt-Wishart matrix in the  $p$-tensor setting?
  \end{itemize}

\end{document}